\documentclass[12pt]{amsart}

\font\bbbld=msbm10 scaled\magstephalf
\newcommand{\bM}{\bar{M}}

\newcommand{\bfR}{\hbox{\bbbld R}}

\newcommand{\tr}{\mbox{tr}}

\newcommand{\tF}{\tilde{F}}
\newcommand{\tU}{\tilde{U}}

\newcommand{\tm}{\tilde{m}}

\newcommand{\ve}{{\bf e}}

\newcommand{\ol}{\overline}
\newcommand{\ul}{\underline}

\newtheorem{theorem}{Theorem}[section]
\newtheorem{lemma}[theorem]{Lemma}
\newtheorem{proposition}[theorem]{Proposition}
\newtheorem{corollary}[theorem]{Corollary}

 \theoremstyle{definition}

\theoremstyle{remark}
\newtheorem{remark}[theorem]{Remark}

\numberwithin{equation}{section}

\begin{document}

\title[fully nonlinear parabolic equations]
{%The Dirichlet Problem for 
%The Initial-Boundary Value Problems for \\
On Estimates for Fully Nonlinear Parabolic Equations
on Riemannian Manifolds
%in Bounded Domains
}
\author{Bo Guan}
\address{Department of Mathematics, Ohio State University,
               Columbus, OH 43210, USA}
\email{guan@math.osu.edu}
\thanks{The first and second authors were supported in part by NSF grants and a scholarship from China Scholarship Council, respectively.}
%\date{}
\author{Shujun Shi}
\address{School of Mathematical Sciences, Harbin Normal University, Harbin 150025, China}
\email{shjshi@163.com}
\author{Zhenan Sui}
\address{Department of Mathematics, Ohio State University,
               Columbus, OH 43210, USA}
\email{sui@math.osu.edu}

\begin{abstract}
In this paper we present some new ideas to derive {\em a priori}
second order estiamtes for a wide class of fully nonlinear parabolic equations. 
Our methods, which produce new existence results for the
initial-boundary value problems in $\bfR^n$, are powerful enough 
to work in general Riemannian manifolds.



\end{abstract}

\maketitle

\section{Introduction}
\label{3I-I}

\medskip

In this paper we are concerned with 
%find technical tools to overcome %substantial difficulties 
deriving {\em a priori} second order estimates for 
fully nonlinear parabolic equations on Riemannian manifolds. 
Let $M^n$ be a compact Riemannian manifold of dimension $n \geq 2$
with smooth boundary $\partial M$ which may be empty ($M$ is closed). 
Let $\chi$ be a smooth $(0, 2)$ tensor on 
$\bM = M \cup \partial M$ and $f$ a smooth symmetric function of $n$ 
variables. 
We consider the fully nonlinear parabolic equation
\begin{equation} 
\label{eq1-1}
f (\lambda(\nabla^2 u + \chi)) = e^{u_t + \psi}
\;\; \mbox{in $M \times \{t > 0\}$},
\end{equation}
where $\nabla^2 u$ denotes the spatial 
Hessian of $u$, $u_t = \partial u/\partial t$, 
and $\lambda (A) = (\lambda_1, \ldots, \lambda_n)$ 
will be the eigenvalues of a $(0,2)$ tensor $A$; 
throughout the paper we shall use $\nabla$ to denote the Levi-Civita 
connection of $(M^n g)$, and assume 
$\psi \in C^{\infty} (\bM \times \{t \geq 0\})$.

The corresponding ellitpic equations were first studied by 
Caffarelli, Nirenberg and Spruck~\cite{CNS3} in $\bfR^n$, as well as in 
 \cite{ChouWang01},\cite{Guan94}, \cite{Guan99a}, \cite{Guan14}, \cite{Guan}, 
\cite{GJ}, %\cite{Ivochkina91}, 
\cite{ITW04}, \cite{LiYY90},
%\cite{Trudinger90} 
\cite{Trudinger95} and \cite{Urbas02} etc.
Following \cite{CNS3}, 
%who first studied 
we assume $f$ to be defined in an open symmetric convex 
cone $\Gamma \subset \bfR^n$ with vertex at origin, 
$\Gamma_n := \{\lambda \in \bfR^n: \lambda_i > 0, 
\; \forall \, 1 \leq i \leq n\} \subseteq \Gamma$,
and to satisfy the
fundamental structure conditions
which have become standard in the literature:
\begin{equation}
\label{eq1-3} 
f_i = f_{\lambda_i} \equiv \frac{\partial f}{\partial
\lambda_i} > 0 \;\; \mbox{in $\Gamma$}, \;\; 1 \leq i \leq n,
\end{equation}
and %so that Eq.~\eqref{eq1-1} is parabolic for solutions 
\begin{equation}
\label{eq1-4} 
\mbox{$f$ is a concave function in $\Gamma$}.
\end{equation}

Equation~\eqref{eq1-1} is parabolic for a solution $u$
% \in C^{2,1} (M_T)$ 
with $\lambda [u] :=\lambda (\nabla^2 u + \chi) \in \Gamma$ 
for $x \in M$ and $t > 0$ (see \cite{CNS3}); 
we shall call such functions {\em admissible}. It is uniformally parabolic 
if $\lambda [u]$ falls in a compact subset of $\Gamma$ and, on the
other hand, may become degenerate 
if $\lambda [u] \in \bar{\Gamma} = \Gamma \cup \partial \Gamma$.
To prevent the degeneracy we shall need the following condition
\begin{equation} 
\label{eq1-5}
\sup_{\partial \Gamma} f := \sup_{\lambda_0 \in \partial \Gamma} \lim_{\lambda \rightarrow \lambda_0} f (\lambda) \leq 0.
\end{equation}
In addition, we shall assume that $f$ is unbounded from above. 
 In particular,
\begin{equation} 
\label{eq1-0}
\lim\limits_{R \rightarrow \infty} f (R {\bf 1}) = \infty.
\end{equation}
where and hereafter %and throughout the paper,  
${\bf 1} = (1, \ldots, 1)$.

Throughout the paper,  let $\varphi^b \in C^{\infty} (\bM)$ with 
%$\lambda [\varphi^b] \in \Gamma$ in $\bM$ 
\begin{equation}
\label{eq1-31} 
\lambda [\varphi^b] \in \Gamma, \;\; f (\lambda [\varphi^b]) > 0 \;\; \mbox{in $\bM$} 
\end{equation}
and,  when $\partial M \neq \emptyset$, 
$\varphi^s \in C^{\infty} (\partial M \times \{t \geq 0\})$. 
By the short time existence theorem, 
there exists a unique admissible solution 
$u \in C^{\infty} (\bM \times (0, t_0]) \cap  C^0 (M \times [0, t_0])$, for some $t_0 > 0$, of
equation~\eqref{eq1-1} satisfying the initial boundary value conditions
\begin{equation}
\label{eq1-1IB} 
u|_{t=0} = \varphi^b \;\; \mbox{in $\bM$}, \;\; 
u = \varphi^s \;\; \mbox{on $\partial M \times \{t > 0\}$}.
\end{equation}
Moreover, $u \in C^{\infty} (\bM \times [0, t_0])$ if 
%$\varphi^s|_{t=0} = \varphi^b$ on $\partial M$ and 
the following compatibility conditions are satisfied
\begin{equation} 
\label{eq1-8}
  f (\lambda [\varphi^b]) = e^{\varphi^s_t + \psi}, \;\;
   \varphi^s = \varphi^b
 \;\; \mbox{on $\partial M \times \{t = 0\}$}.
\end{equation}

Our primary goal in this paper is to establish second order estimates
for admissible solutions of the initial-boundary value problem~\eqref{eq1-1} 
and \eqref{eq1-1IB}. 
Without loss of generality, we may assume \eqref{eq1-8} is satisfied.
For we only have to consider a new initial time, say
$t = t_0/2$ in place of $t = 0$, if necessary. 

For $T > 0$ let 
\[ M_T = M \times (0, T], \;\; \bM_T = \bM \times (0, T] \]
and let $\partial M_T := \partial_s M_T \cup \partial_b M_T$ be 
the parabolic boundary of $M_T$
where 
\[ \partial_s M_T = \partial M \times [0, T), \;\; 
    \partial_b M_T = \bM \times \{ t=0\}. \]
So $\partial M_T = \partial_b M_T$ when $M$ is closed.
Let %$u \in C^{\infty} (\ol{M_T})$
$u \in C^{4, 2} (M_T) \cap C^{2, 1} (\ol{M_T})$ 
be an admissible solution of the problem~\eqref{eq1-1} and 
\eqref{eq1-1IB}. 
%satisfying the parabolic boundary condition
We wish to establish the {\em a priori} estimate
\begin{equation}
\label{eq1-C2} 
|\nabla ^2 u| \leq C \;\; \mbox{in $\ol{M_T}$}.
\end{equation}
%where $C$ may depend on $T$,  $|u|_{C^{1, 0} (\ol{M_T})}$, 
%$|\ul u|_{C^{2,1} (\ol{M_T})}$ and other known data.

As our first main result in this paper we derive \eqref{eq1-C2}
assuming the existence of an admissible subsolution. 

\begin{theorem}
\label{thm2} 
In addition to conditions \eqref{eq1-3}-\eqref{eq1-0}, suppose that 
there exists an admissible
subsolution $\ul{u}\in C^{2,1}(\ol {M_T})$ satisfying
\begin{equation}
\label{eq1-11}
%\left\{
\begin{aligned}
%\frac{\partial \ul u}{\partial t}  
  f(\lambda [\ul u]) \geq  e^{\ul u_t + \psi}
         %\nabla^2\ul{u} + \chi)) - \psi
                                              \;\; \mbox{in $M_T$}  \\
 % \ul u \leq \,& \varphi  \;\; \mbox{on $\partial_b M_T$}, \\
 % \ul u = \,& \varphi  \;\; \mbox{on $\partial_s M_T$} .
 \end{aligned} %\right.
\end{equation}
and the initial-boundary conditions
\begin{equation} 
\label{eq1-11b}
\left\{
\begin{aligned}
  \ul u \leq \,& \varphi^b  \;\; \mbox{on $\partial_b M_T$}, \\
      \ul u = \,& \varphi^s \;\; \mbox{on $\partial_s M_T$} .
 \end{aligned} \right.
\end{equation}
Then
\begin{equation}
\label{eq1-12} 
\sup_{{M_T}} |\nabla^2 u| \leq
 C_1 + C_1 \max_{\partial M_T}|\nabla^2 u|
\end{equation}
In particular, \eqref{eq1-C2} holds when $M$ is closed. 

Suppose moreover that for any $b > a> 0$, there exists
$K_1 \geq 0$ such that 
\begin{equation}
\label{eq1-13} 
\sum f_i (\lambda) \lambda_i \geq  - K_1 \Big(1 +
\sum f_i\Big)  \;\; \mbox{in $\Gamma^{[a, b]} := 
\{\lambda \in \Gamma:  a \leq f (\lambda) \leq b\}$}.
\end{equation}
 Then 
\begin{equation}
\label{eq1-14} 
\max_{\partial M_T}|\nabla^2 u| \leq C_2.
\end{equation}
 \end{theorem}

\begin{remark}
\label{remark-1.1}
In Theorem~\ref{thm2} and the rest of  this paper, unless otherwise 
indicated the constant $C_1$ in \eqref{eq1-12} will depend on 
\begin{equation} 
\label{eq1-81}
 |u|_{C^1 (\ol{M_T})}, \; |\psi|_{C^{2,1} (\ol{M_T})}, \;
|\ul u|_{C^{2,1} (\ol{M_T})}, \;
\inf_{{M_T}} \mbox{dist} (\lambda [\ul u], \partial \Gamma), 
\end{equation}
as well as geometric quantities of $M$, while  
$C_2$ in \eqref{eq1-14} will depend in addtion on 
$|\varphi^b|_{C^2 (\bM)}$,
$|\varphi^s|_{C^{4,1} (\partial_s M_T)}$ and geometric quantities of 
$\partial M$.
\end{remark}
 
\begin{remark}
%\label{remark-1.1}
The proof of \eqref{eq1-12} does not need assumptions \eqref{eq1-5} and 
\eqref{eq1-11b}. This will be clear in Section~\ref{3I-C}. 
For the boundary estimate \eqref{eq1-14}, we need
condition~\eqref{eq1-5} to prevent equation~\eqref{eq1-1} from 
being degenerate along the boundary.  
It would be interesting to establish \eqref{eq1-14} in the degenerate case. 
We also expect Theorem~\ref{thm2} to hold without 
conditions~\eqref{eq1-0} and \eqref{eq1-13} which are fairly mild and technical 
in nature. When $M$ is a bounded smooth domain in $\bfR^n$
these assumptions can be removed. % from Theorem~\ref{thm2}. 
\end{remark}

\begin{remark}
\label{remark-1.3}
If we replace \eqref{eq1-0} by the assumption
\begin{equation}
\label{eq1-01}
\lim_{ |\lambda|\to \infty}
|\lambda|^2\sum f_i = \infty,
\end{equation}
then $C_1$ in \eqref{eq1-12} can be chosen independent of 
$|u_t|_{C^0(\ol{M_T})}$; see Remark \ref{eq1-01a}.
\end{remark}

Our next result concerns \eqref{eq1-12}  under a new condition
 which is
optimal in many cases and is in general weaker than the subsolution assumption in Theorem~\ref{thm2}, especially on closed manifolds.  
It is motivated by recent work in \cite{Guan14}.
%We shall deal with equation~\eqref{eq1-1a} directly.

For $\sigma \in \bfR$ define
\[ \Sigma^{\sigma} := \{(\lambda, z)  \in \Gamma \times \bfR: 
                                f (\lambda) > e^{z + \sigma}\}
           %= \{(\lambda, z) \in \Gamma \times \bfR: \lambda \in \Gamma^{s+\sigma}\} 
\]
and let $\partial \Sigma^{\sigma}$ be the boundary of $\Sigma^{\sigma}$.
By \eqref{eq1-3} and \eqref{eq1-4}, 
$\partial \Sigma^{\sigma}$ is a smooth convex hypersurface 
in $\Gamma \times \bfR$.  
For $\hat{\lambda} = (\lambda, z) \in \partial \Sigma^{\sigma}$
let 
\[ \nu_{\hat{\lambda}} 
   = \frac{(Df (\lambda), -f (\lambda))}{\sqrt{f (\lambda)^2 + |Df (\lambda)|^2}} \]
denote the unit normal vector to $\partial \Sigma^{\sigma}$ at $\hat{\lambda}$. 
Finally, for $\hat \mu \in \Gamma \times \bfR$ let
\[ \hat S^{\sigma}_{\hat{\mu}} 
       := \{\hat{\lambda} \in \partial \Sigma^{\sigma}: 
(\hat{\mu} - \hat{\lambda}) \cdot \nu_{\hat{\lambda}} \leq 0\}. \]

\begin{theorem}
\label{thm3} 
Under conditions \eqref{eq1-3} and \eqref{eq1-4}, 
the estimate \eqref{eq1-12} holds provided that
there exists an admissible function 
$\ul{u}\in C^{2,1}(\ol {M_T})$ satisfying 
%for any $(x, t) \in M_T$ and $R  > \ul u_t (x, t)$, 
\begin{equation} 
\label{eq1-11a}
%\{\hat{\lambda} := (\lambda, z) \in \partial \Sigma^{\psi (x, t)}: 
 %       (\hat{\mu} - \hat{\lambda}) \cdot \nu_{\hat{\lambda}} \leq 0, 
 %       \; \ul u_t (x, t) - \delta \leq z \leq R \} 
 \hat S^{\psi (x, t)}_{\hat{\mu}} \cap 
  \Gamma \times [a, b] 
\;\; \mbox{is compact}, \; \forall \, (x, t) \in M_T,  \; 
   \forall \, [a, b] \subset \bfR 
%  \lambda [\ul u] \in \mathcal{B}_{\ul u_t + \psi}^+   \;\; \mbox{in $M_T$}.
\end{equation}
 where $\hat{\mu} = (\lambda [\ul u (x, t)], \ul u_t (x, t))$.
\end{theorem}

By the concavity of $f$, 
if $\ul u$ is an admissible subsolution then 
$(\hat{\mu} - \hat{\lambda}) \cdot \nu_{\hat{\lambda}} \geq 0$
for any $\hat{\lambda} \in \Sigma^{\psi (x, t)}$.

\begin{remark}
\label{remark-1.10}
In Theorems~\ref{thm2} and ~\ref{thm3}, the constants $C_1$ and $C_2$  depend on 
$T$ only implicitly. For instance, if the quantities listed in \eqref{eq1-81}
%(listed in Theorem~\ref{thm2}) 
%which $C_1$ or $C_2$ depends on in Theorem~\ref{thm2}
are all independent of $T$, then so is $C_1$.
The independence on $T$ of the estimates is important to understnding the asymptotic 
behaviors of solutions as $t$ goes to infinity. 
If one allows $C_1$ to depend on $T$ (explicitly), \eqref{eq1-12} can be derived 
under much weaker conditions, and more easily. 
\end{remark}

\begin{theorem}
\label{thm2a} 
Under assumptions~\eqref{eq1-3}, \eqref{eq1-4} and \eqref{eq1-31}, 
\begin{equation}
\label{eq1-12a} 
|\nabla^2 u (x, t)| \leq
 C e^{B t} \Big(1 + \max_{\partial M_T}|\nabla^2 u|\Big), 
   %\max_{\bM} f (\lambda [\varphi^b])\}\Big),  
\;\; \forall \, (x, t) \in M_T
\end{equation}
where $C$ and $B$ depend on $|\nabla u|_{C^{0} (\ol{M_T})}$, 
$|\varphi^b|_{C^2 ({\bM})}$ and other known data. 
In particular, if $M$ is closed then 
$|\nabla^2 u (x, t)| \leq C e^{B t}$. 
   %\exp \{t \max_{\bM} f (\lambda [\varphi^b])\}$.
\end{theorem}

%Let $A = \max_{\bM} f (\lambda [\varphi^b])$
Note that by \eqref{eq1-31}  the function 
\[ \ul u := \varphi^b +  t \min_{\bM} \{\log f (\lambda [\varphi^b]) - \psi\} \]
is admissible and satisfies \eqref{eq1-11}. 

An immediate consequence of Theorem~\ref{thm2a} is the following 
characterization of finite time blow-up solutions on closed manifolds.

\begin{corollary}
\label{thm2ac1} 
Assume $M$ is closed and $f$ satisfies \eqref{eq1-3}-\eqref{eq1-5}.
Then equation~\eqref{eq1-1} admits a unique admissible solution
$u \in C^{\infty} (M \times \bfR^+)$ with 
initial value function $\varphi^b$ satisfying \eqref{eq1-31}, 
provided that the {a priori} gradient estimate holds
\begin{equation}
\label{eq1-12g} 
\sup_{M_T} |\nabla u| \leq C,  \;\; \forall \, T > 0
\end{equation}
%holds for all $T > 0$. 
where $C$ may depend on $T$. 
In other words, if $u$ has a finite time blow-up at $T < \infty$,
then 
\[ \lim_{t \rightarrow T^-} \max_{x \in M} |\nabla u (x, t)| =  \infty. \]
\end{corollary}

So the long time existence of solutions in $0 \leq t < \infty$
reduces to establishing 
gradient estimate \eqref{eq1-12g}. 
This is also true 
%for the initial-boundary value problem equation~\eqref{eq1-1}
%and \eqref{eq1-11b} 
when $\partial M \neq \emptyset$. Using Theorem~\ref{thm2} 
%and the gradient estimates in Section~\ref{3I-G} 
we can prove the following existence results.

\begin{theorem}
\label{thm1}
 Assume \eqref{eq1-3}-\eqref{eq1-31},   \eqref{eq1-13}, 
and \eqref{eq1-11}-\eqref{eq1-11b} hold for $T \in (0, \infty]$.
 Then there exists a unique admissible solution
$u \in C^{\infty} (\bM_T) \cap C^0 (\ol {M_T})$
of equation \eqref{eq1-1} satsfying \eqref{eq1-1IB}, provided that 
any one of the
following conditions holds: ({\sl i}) $\Gamma = \Gamma_n$; 
({\sl ii}) $(M, g)$ has nonnegative sectional curvature; 
({\sl iii}) there is $\delta_0 > 0$ such that
\begin{equation}
\label{eq1-20}
  f_j \geq \delta_0 \sum f_i \;\; \mbox{if $\lambda_j < 0$,
on $\partial \Gamma^{\sigma} \;\; \forall \; \sigma > 0$};
\end{equation} 
and ({\sl iv}) $K_1 = 0$ in \eqref{eq1-13} and $\nabla^2 w \geq \chi$
for some function $w \in C^{2}(\bM)$.
\end{theorem}

The assumptions ({\sl i})-({\sl iv}) are only needed in deriving the gradient estimates. It would be interesting to remove these assumptions.
When $\partial M = \emptyset$, Theorem~\ref{thm1} holds without
assumptions~\eqref{eq1-0} and \eqref{eq1-11}-\eqref{eq1-11b}, 
and condition~\eqref{eq1-13}
can be removed in each of the cases ({\sl i})-({\sl iii}).

Theorem~\ref{thm1}  applies to a very general class of equations
including $f = \sigma_k^{1/k}$ and 
$f = (\sigma_k/\sigma_l)^{1/(k-l)}$, $1 \leq l < k \leq n$ where 
$\sigma_k$ is the $k$-th elementry symmetric function defined on the 
cone $\Gamma_k := \{\lambda \in \bfR^n: \sigma_j (\lambda) > 0, 
\forall, 1 \leq j \leq k\}$. 
Another interesting example is $f = \log P_k$ to which 
Theorem~\ref{thm1} applies,  where
\[ P_k (\lambda) := \prod_{i_1 < \cdots < i_k} 
(\lambda_{i_1} + \cdots + \lambda_{i_k}), \;\; 1 \leq k \leq n\]
defined in the cone
\[ \mathcal{P}_k : = \big\{\lambda \in \bfR^n:
      \lambda_{i_1} + \cdots + \lambda_{i_k} > 0, 
\; \forall \, 1 \leq i_1 < \cdots < i_k \leq n\big\}. \]

\begin{corollary}
\label{thm1c1} 
Let $f = (\sigma_k/\sigma_l)^{1/(k-l)}$, $\Gamma = \Gamma_k$, ($0 \leq l < k \leq n$; $\sigma_0 = 1$), or  $f = \log P_k$ and 
$\Gamma = \mathcal{P}_k$. The parabolic problem 
\eqref{eq1-1} and \eqref{eq1-1IB} with smooth data admits
a unique admissible solution
$u \in C^{\infty} (\bM_T) \cap C^0 (\ol {M_T})$, provided that there exists an admissible subsolution $\ul u \in C^{2,1} (\ol{M_T})$ satisfying
\eqref{eq1-11}-\eqref{eq1-11b}.
\end{corollary}

For $f = (\sigma_k/\sigma_l)^{1/(k-l)}$ or  $f = \log P_k$, 
an admissible subsolution satisfies \eqref{eq1-11a}; see \cite{Guan14}.
Except for $f = \sigma_k^{1/k}$, Corollary~\ref{thm1c1}  is new 
even when $M$ is a bounded smooth domain in $\bfR^n$; see also \cite{JS14}.
On the other hand, for a bounded smooth domain in $\bfR^n$, %Theorem~\ref{thm1} 
we have the following result which is essentially optimal, both
in terms of assumptions on $f$ and the generality of the domain.

\begin{theorem}
\label{thm1c}
Let $M$ be a bounded smooth domain in $\bfR^n$,
$0 < T \leq \infty$,  and 
let $\chi = \{\chi_{ij}\}$ be a symmetric matrix with 
$\chi_{ij} \in C^{\infty} (\ol{M_T})$.   
Under conditions \eqref{eq1-3}-\eqref{eq1-31}  
and \eqref{eq1-11}-\eqref{eq1-11b},
there exists a unique admissible solution
$u \in C^{\infty} (\bM_T) \cap C^0 (\ol {M_T})$
of equation \eqref{eq1-1} satsfying \eqref{eq1-1IB}.
\end{theorem}

 The first initial-boundary value problem for equation \eqref{eq1-1} 
or \eqref{eq1-1a} in $\bfR^n$ was treated by 
Ivochinkina-Ladyzhenskaya~\cite{IL95}, \cite{IL97},  and by
Wang~\cite{Wang94}, Chou-Wang~\cite{ChouWang01} for 
$f = (\sigma_k)^{1/k}$; see also \cite{Lieberman}. 
Jiao-Sui~\cite{JS14} recently studied equation~\eqref{eq1-1a} 
on Riemannian manifolds under additional assumptions.

The rest of the article is devided into three sections. In 
Sections \ref{3I-C} and \ref{3I-B}  
we derive \eqref{eq1-12} and \eqref{eq1-14} 
respectively, completing the proofs of 
Theorems~\ref{thm2}, \ref{thm3} and \ref{thm2a}.
Instead of ~\eqref{eq1-1}, we shall deal with the equation
\begin{equation} 
\label{eq1-1a}
 %\frac{\partial u}{\partial t} = 
f (\lambda(\nabla^2 u + \chi)) = u_t + \psi
\end{equation}
under essentially the same assumptions on $f$ with the exception 
that \eqref{eq1-5} is replaced by
\begin{equation} 
\label{eq1-9}
\inf_{\partial_s M_T} (\varphi_t + \psi) - \sup_{\partial \Gamma} f > 0
\end{equation}
which is need in the proof of \eqref{eq1-14} .
Accordingly,  the functions $\varphi^b$ and $\ul u \in C^{2,1} (\ol{M_T})$ 
are assumed to satisfy  $\lambda [\varphi^b] \in \Gamma$
in $\bM$ and, respectivley, 
\begin{equation}
\label{eq1-11'}
%\left\{
\begin{aligned}
%\frac{\partial \ul u}{\partial t}  
  f(\lambda [\ul u]) \geq  {\ul u_t + \psi}
         %\nabla^2\ul{u} + \chi)) - \psi
                                              \;\; \mbox{in $M_T$}  \\
 % \ul u \leq \,& \varphi  \;\; \mbox{on $\partial_b M_T$}, \\
 % \ul u = \,& \varphi  \;\; \mbox{on $\partial_s M_T$} .
 \end{aligned} %\right.
\end{equation}
in place of \eqref{eq1-11}.
Note that if $f > 0$ in $\Gamma$ and satisfies \eqref{eq1-3},  
\eqref{eq1-4}, \eqref{eq1-0} and \eqref{eq1-13} 
then the function $\log f$ still satisfies theses assumptions. 
So equation~\eqref{eq1-1} is covered by \eqref{eq1-1a} in most cases,  
and we shall derive the estimates for equation~\eqref{eq1-1a}.
 
 In Section~\ref{E} we briefly discuss the proof of the
existence results and the preliminary estimates needed in the proof.

At the end of this Introduction we recall the following commonly used
notations
\[ \begin{aligned}
  |u|_{C^{k,l} (\ol{M_T})} = \,& \sum_{j=0}^k |\nabla^j u|_{C^0 (\ol{M_T})}
+\sum_{j=1}^l \Big|\frac{\partial^j u}{\partial t^j}\Big|_{C^0(\ol{M_T})}, \\
 |u|_{C^{k+\alpha, l+\beta} (\ol{M_T})} =\,& |u|_{C^{k,l} (\ol{M_T})}
 + |\nabla^k u|_{C^\alpha (\ol{M_T})} 
 + \Big|\frac{\partial^l u}{\partial t^l}\Big|_{C^{\beta} (\ol{M_T})} 
  \end{aligned} \]
where $0 < \alpha, \beta < 1$ and $k, l = 1, 2, \ldots$,
for a function $u$ sufficiently smooth on $\ol{M_T}$.
We shall also write $|u|_{C^{k} (\ol{M_T})} = |u|_{C^{k,k} (\ol{M_T})}$. 

\hspace{0.3cm}

\noindent
\textbf{Acknowledgement.} 
Part of this work was done while the second author was visiting
Department of Mathematics at Ohio State University. He wishes to thank 
the Department and University for their hospitality.

\bigskip

\section{Global estimates for second derivatives}
\label{3I-C}
\setcounter{equation}{0}

\medskip

A substantial difficulty in deriving the global estimate \eqref{eq1-12},
which is our primary goal in this section,
% in Theorems~\ref{thm2} and \ref{thm3}. 
 is caused due to the presense of curvature of $M$; another is the lack of (globally
defined) functions or geometric quantities with desirable properties. 
In our proof the use of the function $\ul u$, which is either an
admissible subsolution as in Theorem~\ref{thm2} or satisfies \eqref{eq1-11a}, is critical.
%See \cite{Guan14} and \cite{Guan} for work in the ellpitic case.
%As indicated at the end of last section 
We shall consider equation~\eqref{eq1-1a} in place of \eqref{eq1-1}.
%in this and following sections.

Let $u \in C^{4,2} (M_T) \cap C^{2,1} (\ol{M_T})$ be an admissible solution of \eqref{eq1-1a}, and $\ul u \in C^{2,1} (\ol{M_T})$ an
admissible function. We assume 
that $u$ admits an {\em a priori} $C^1$ bound
\begin{equation}
\label{eq2-195}
|u|_{C^1 (\ol {M_T})} \leq C.
%:= |u|_{C^0 (\ol {M_T})} + |\nabla u|_{C^0 (\ol {M_T})} + |u_t|_{C^1 (\ol {M_T})}  \leq C
\end{equation}

Let $\phi (s) = - \log (1 - b s^2)$ and
%\[ \eta = - \log (1 - b |\nabla (u - \ul u)|^4) + a (\ul{u} - u - t) \]
\begin{equation}
\label{eq2-200}
 \eta = \phi (1 +|\nabla (u - \ul u)|^2) + a (\ul{u} - u - \delta t)
\end{equation}
where $a, b, \delta > 0$ are constants and $\ul u \in C^{2,1} (\ol {M_T})$
is an admissible function;
we shall choose $\delta = 1$ or $0$,
%and $a, b$ are positive constants to be determined;
$a$ sufficiently large
while $b$ small enough,
\begin{equation}
\label{eq2-205}
b \leq \frac{1}{8 b_1^2}, \;\;  
   b_1 = 1 + \sup_{{M_T}} |\nabla  (u - \ul u)|^2. 
\end{equation}

Consider the quantity
\[ W = \sup_{(x,t) \in {M_T}} \max_{\xi \in T_x M^n, |\xi| = 1}
  (\nabla_{\xi \xi} u + \chi (\xi, \xi)) e^{\eta}. \]
%where 
%\begin{equation}
%\label{eq2-200}
% \eta = \phi (1 +|\nabla (u - \ul u)|^2) + a (\ul{u} - u - \delta t) 
%\eta = - \log (1 - b (1 + |\nabla (u - \ul u)|^2)^2) 
%          + a (\ul{u} - u - \delta t)
%\end{equation}
Suppose $W$ is
achieved at an interior point $(x_0, t_0) \in  M_T$ for a unit
vector $\xi \in T_{x_0} M^n$. Let $e_1, \ldots, e_n$ be
smooth orthonormal local frames about $x_0$ such that $e_1 = \xi$,
$\nabla_i e_j = 0$ and $U_{ij} := \nabla_{ij} u + \chi_{ij}$ are diagonal 
at $(x_0, t_0)$. 
So $W=U_{11} (x_0, t_0) e^{\eta (x_0, t_0)}$. 
We wish to derive a bound 
\begin{equation}
\label{eq2-148} 
U_{11} (x_0, t_0)  \leq C. 
\end{equation}

Write equation \eqref{eq1-1a} in the form
\begin{equation}
\label{eq1-1-1}
u_t = F (U) - \psi, \;\;  U = \{U_{ij}\}
\end{equation}
where $F$ is defined by 
\[ F (A) \equiv f (\lambda [A]) \]
for an $n \times n$ symmetric
matrices $A = \{A_{ij}\}$ with eigenvalues $\lambda [A] \in \Gamma$.
Differentiating \eqref{eq1-1-1} gives
\begin{equation}
\label{eq2-4} 
\begin{aligned}
              u_{tt} = \,& F^{ij} U_{ijt} - \psi_t,\\
   \nabla_k u_t = \,& F^{ij} \nabla_k U_{ij} - \nabla_k \psi,
                               \;\; \forall \, k, \\
%\end{equation}
%\begin{equation}
%\label{eq2-5}
%\begin{aligned}
\nabla_{11} u_t = \,& F^{ij} \nabla_{11} U_{ij} 
                                 + F^{ij, kl} \nabla_1 U_{ij} \nabla_1 U_{kl}
                                 - \nabla_{11}\psi.
\end{aligned}
\end{equation}
Throughout the paper we use the notation
\[ F^{ij} = \frac{\partial F}{\partial A_{ij}} (U), \;\;
  F^{ij, kl} = \frac{\partial^2 F}{\partial A_{ij} \partial A_{kl}} (U). \] 
The matrix $\{F^{ij}\}$ has eigenvalues $f_1, \ldots, f_n$, and therefore is positive definite when $f$ satisfies \eqref{eq1-3}, 
while \eqref{eq1-4}
implies that $F$ is a concave function; see \cite{CNS3}.
Moreover, the following identities hold
\[   \begin{aligned}
F^{ij} U_{ij} = \,& \sum f_i \lambda_i, \;\;
  F^{ij} U_{ik} U_{kj} = \sum f_i \lambda_i^2. 
  \end{aligned} \]
We also note that $F^{ij}$ are diagonal at $(x_0, t_0)$.

\begin{proposition}
\label{eq2-149}
For any $a, C_1 > 0$ there exists a constant $b > 0$ 
satisfying \eqref{eq2-205} 
such that,  at $(x_0, t_0)$, if $U_{11} \geq C_1 a/b$ then 
\begin{equation}
\label{eq2-141}
 \begin{aligned}
\frac{b}{2} F^{ii} U_{ii}^2 + a F^{ii} \nabla_{ii} (\ul u - u) - a (\ul u_t - u_t)  + a \delta
      \leq C \sum F^{ii} + C.
\end{aligned}
\end{equation}
\end{proposition}

\begin{proof}
We shall assume $U_{11} (x_0, t_0) \geq 1$.
At $(x_0, t_0)$ where the function $\log U_{11}+\eta$ has
its maximum, 
\begin{equation} 
\label{eq2-1}
\frac{(\nabla_{11} u)_t}{U_{11}} + \eta_t \geq 0, \;\; 
\frac{\nabla_{i} U_{11}}{U_{11}} + \nabla_{i} \eta = 0, \;\; 1 \leq i \leq n,
\end{equation}
and
\begin{equation}
\label{eq2-2}
\begin{aligned}
  \frac{1}{U_{11}} F^{ii} \nabla_{ii} U_{11}
   - \frac{1}{U_{11}^2} F^{ii} (\nabla_i U_{11})^2 
   + F^{ii} \nabla_{ii} \eta \leq 0.
\end{aligned}
\end{equation}

%Throughout the paper we write $F^{ij} = F^{ij} (\{U_{ij}\})$.
From the identity %an equality on Riemannian manifold
\begin{equation}
\label{eq2-6}
\begin{aligned}
\nabla_{ijkl} v - \nabla_{klij} v
\,& = R^m_{ljk} \nabla_{im} v + \nabla_i R^m_{ljk} \nabla_m v
      + R^m_{lik} \nabla_{jm} v \\
  + \,& R^m_{jik} \nabla_{lm} v
      + R^m_{jil} \nabla_{km} v + \nabla_k R^m_{jil} \nabla_m v
\end{aligned}
\end{equation}
it follows that 
%By  (\ref{hess-a65}) and (\ref{hess-A80}),
\begin{equation}
\label{eq2-7}
\begin{aligned}
F^{ii} \nabla_{ii} U_{11}
\geq \,& F^{ii} \nabla_{11} U_{ii} - C  U_{11} \sum F^{ii}，
% \geq \,& F^{ii} \nabla_{11} U_{ii} - C (|\nabla u| + U_{11}) \sum F^{ii}.
\end{aligned}
\end{equation}
where $C$ depends on
$|\nabla u|_{C^0 (\bM_T)}$ and geometric quantities of $M$. 
By \eqref{eq2-2}, \eqref{eq2-7} and \eqref{eq2-4} we obtain
\begin{equation}
\label{eq2-8}
\begin{aligned}
F^{ii} \nabla_{ii} \eta - \eta_t \leq \,& 
   \frac{1}{U_{11}} F^{ij, kl} \nabla_1 U_{ij} \nabla_1 U_{kl}
           + \frac{1}{U_{11}^2} F^{ii} (\nabla_i U_{11})^2 \\
    &  - \frac{\nabla_{11} \psi}{U_{11}} + C \sum F^{ii}.
\end{aligned}
\end{equation}
%where
%\[ E \equiv  F^{ij, kl} \nabla_1 U_{ij} \nabla_1 U_{kl}
%           + \frac{1}{U_{11}} F^{ii} (\nabla_i U_{11})^2. \]

%we follow the idea of Urbas~\cite{Urbas02}.
Let  %$0 < s < 1$  (to be chosen) and
\[  \begin{aligned}
J \,& = \{i: 3 U_{ii} \leq - U_{11}\}, \;\;
%J & = \{U_{ii} > - a U_{11}, \; F^{ii} > a^{-1} F^{11}\}, \\
%K & = \{i: U_{ii} > a U_{11}, \; F^{ii} < a^{-1} F^{11}, \; i \neq 1\}, \\
K  = \{i> 1: 3 U_{ii} > - U_{11}\}.
  \end{aligned} \]
As in \cite{Guan14}, which uses an idea of Urbas~\cite{Urbas02}, 
one derives
\begin{equation}
\label{eq2-10}
\begin{aligned}
F^{ii} \nabla_{ii} \eta - \eta_t 
  \leq \,& \sum_{i \in J} F^{ii} (\nabla_i \eta)^2
             + C F^{11} \sum_{i \notin J} (\nabla_i \eta)^2
     - \frac{\nabla_{11} \psi}{U_{11}} + C \sum F^{ii}.
 %E \leq \,& U_{11} \sum_{i \in J} F^{ii} (\nabla_i \eta)^2
 %            + C U_{11} F^{11} \sum_{i \notin J} (\nabla_i \eta)^2
  %              + \frac{C(1+ |\nabla u|^2)}{U_{11}} \sum F^{ii}.
\end{aligned}
\end{equation}

For convenience we write $w = \ul u - u$, $s = 1 +|\nabla w|^2$, 
%$\phi (s) = - \log (1 - b s^2)$,  
and calculate
\[ \begin{aligned}
 \nabla_i \eta
   = \,& 2 \phi' \nabla_{k} w \nabla_{ik} w + a \nabla_i w, \\
  %         = \,& 2 \phi' (U_{ii} \nabla_i u - \chi_{ik} \nabla_{k} u)
   %             + a \nabla_i (\ul{u} - u), \\
 \eta_t = \,& 2\phi' \nabla_k w (\nabla_k w)_t + a w_t - a \delta, \\
  \nabla_{ii} \eta
   = \,&  2 \phi' (\nabla_{ik} w \nabla_{ik} w
          + \nabla_{k} w \nabla_{iik} w)
          + 4 \phi'' (\nabla_{k} w \nabla_{ik} w)^2 %\nabla_l u \nabla_{il} u
          + a \nabla_{ii} w,
\end{aligned} \]
while
\[ \phi' (s) = \frac{2bs }{1 - b s^2}, \;\;
    \phi'' (s) = %\frac{2b}{1 - b t^2} 
              %+ \frac{4b^2 t^2}{ (1 - b t^2)^2}. \]
\frac{2 b + 2 b^2 s^2}{ (1 - b s^2)^2} > 4 (\phi')^2. \] 
Hence,
\begin{equation}
\label{eq2-11}
\begin{aligned}
  \sum_{i \in J} F^{ii} (\nabla_i \eta)^2
    \leq \,& 8 (\phi')^2 \sum_{i \in J} F^{ii} (\nabla_{k} w \nabla_{ik} w)^2
          + 2 |\nabla w|^2 a^2 \sum_{i \in J} F^{ii},
\end{aligned}
\end{equation}
and
\begin{equation}
\label{eq2-12}
 \sum_{i  \notin J} (\nabla_i \eta)^2
\leq C a^2 + C (\phi')^2 U_{11}^2.
\end{equation}
By \eqref{eq2-4},
 \begin{equation}
\label{eq2-13}
 \begin{aligned}
 F^{ii} \nabla_{ii} \eta - \eta_t 
  %= \,&  2 \phi' F^{ii} (\nabla_{ik} u \nabla_{ik} u
  %       + \nabla_{k} u \nabla_{iik} u) \\
  %    &  + 2 \phi'' F^{ii} \nabla_{k} u \nabla_{ik} u \nabla_l u \nabla_{il} u
  %       + A F^{ii} \nabla_{ii} (\ul{u} - u) \\
\geq \,&  \phi'  F^{ii} U_{ii}^2 %+ 2\phi'\nabla_k u \nabla_k \psi 
               + 2 \phi'' F^{ii} (\nabla_{k} w \nabla_{ik} w)^2 \\
         & + a F^{ii} \nabla_{ii} w - a w_t + a \delta
            %(\ul{u} - u) - a (\ul u - u)_t 
          - C \phi' \Big(1 + \sum F^{ii}\Big).
\end{aligned}
\end{equation}
%Note that $\phi'' - 4 (\phi')^2 > 0$. 
It follows from 
%\eqref{hess-a271}, \eqref{hess-a272},\eqref{hess-a272.5} and
 \eqref{eq2-10}-\eqref{eq2-13} that 
\begin{equation}
\label{eq2-14-1}
 \begin{aligned}
\phi' F^{ii} U_{ii}^2 + \,& a F^{ii} \nabla_{ii} w - a w_t + a \delta \\
%(\ul{u} - u) - a (\ul{u}-u)_t \\
    \leq \,& C a^2 \sum_{i \in J} F^{ii}
                 + C (a^2 + (\phi')^2 U_{11}^2) F^{11} 
      - \frac{\nabla_{11} \psi}{U_{11}} + C \Big(\phi' + \sum F^{ii}\Big).
\end{aligned}
\end{equation}
Note that 
\begin{equation}
\label{eq2-21}
 F^{ii} U_{ii}^2 \geq  F^{11} U_{11}^2 + \sum_{i \in J} F^{ii} U_{ii}^2
   \geq F^{11} U_{11}^2 + \frac{U_{11}^2}{9} \sum_{i \in J} F^{ii}.
 \end{equation}
We may fix $b$ small to derive \eqref{eq2-141} when 
%$a$ is sufficiently large and 
$U_{11} \geq C a/b$.
\end{proof}

To proceed we need the following lemma which is key to the proof 
of Theorem~\ref{thm2}, both for \eqref{eq1-12} in this section 
and \eqref{eq1-14} in the next section; compare with Lemma~2.1 in 
\cite{Guan}.

\begin{lemma}
\label{ma-lemma-C10} 
Let $K$ be a compact subset of $\Gamma$ and $\beta > 0$.
There is constant $\varepsilon > 0$ such that, for any $\mu \in K$ and $\lambda \in \Gamma$,  
when $|\nu_{\mu} - \nu_{\lambda}| \geq \beta$
(where $\nu_{\lambda} = Df (\lambda)/|Df (\lambda)|$ 
denotes the unit normal vector to the level surface of $f$ through 
$\lambda$), 
\begin{equation}
\label{gj-C160'}
  \sum f_i (\lambda) (\mu_i - \lambda_i) \geq f (\mu) - f(\lambda) 
         + \varepsilon \Big( 1 + \sum f_i (\lambda)\Big).
\end{equation}
\end{lemma}

\begin{proof} 
Since $\nu_{\mu}$ is smooth in $\mu \in \Gamma$ and $K$ is compact, there is $\epsilon_0 > 0$ such that for any 
$0 \leq \epsilon \leq \epsilon_0$,
\[ K^{\epsilon} :=\{\mu^{\epsilon} := \mu - \epsilon {\bf 1}: 
       \mu \in K\} \] 
is still a compact subset of $\Gamma$ and 
\[ |\nu_{\mu} - \nu_{\mu^{\epsilon}}| \leq \frac{\beta}{2}, 
\;\; \forall \, \mu \in K. \] %\; 0 \leq \epsilon \leq \epsilon_0. \]
%where $\mu^{\epsilon}:= \mu - \epsilon \bf{1}$, 
Consequently, if $\mu \in K$ and $\lambda \in \Gamma$ satisfy  
$|\nu_{\mu} - \nu_{\lambda}| \geq \beta$ then 
$|\nu_{\mu^{\epsilon}} - \nu_{\lambda}| \geq \frac{\beta}{2}$.
 
By the smoothness of the level surfaces of $f$, there exists
$\delta > 0$ (which depends on $\beta$ but is uniform in $\epsilon \in [0, \epsilon_0]$) such that
\[ \min_{\mu \in K} \min_{0 \leq \epsilon \leq \epsilon_0}
\mbox{dist} (\partial B_{\delta}^{\frac{\beta}{2}} (\mu^{\epsilon} ),
\partial \Gamma^{f(\mu^{\epsilon})}) > 0 \]
where $\partial B_{\delta}^{\frac{\beta}{2}} (\mu^{\epsilon})$
denotes the spherical cap
\[ \partial B_{\delta}^{\frac{\beta}{2}} (\mu^{\epsilon})
     = \Big\{\zeta \in \partial B_{\delta} (\mu^{\epsilon}): \;
    \nu_{\mu^{\epsilon}} \cdot  (\zeta - \mu^{\epsilon})/\delta
        \geq \frac{\beta}{2} \sqrt{1-\beta ^2/16}\Big\}. \]
Therefore,
\begin{equation}
\label{gj-C150}
 \theta \equiv \min_{\mu  \in K} 
  \min_{0 \leq \epsilon \leq \epsilon_0}
  \min_{\zeta \in \partial B_{\delta}^{\frac{\beta}{2}} (\mu^{\epsilon})}
    \{f (\zeta) - f(\mu^{\epsilon})\} > 0.
\end{equation}
%that $\theta$ depends on $\beta$.

 Let $P$
be the two-plane through ${\mu}^{\epsilon}$ spanned by
$\nu_{{\mu}^{\epsilon}}$ and $\nu_{{\lambda}}$
 (translated to ${\mu}^{\epsilon}$), and
 $L$ the line on $P$ through
${\mu}^{\epsilon}$ and perpendicular to $\nu_{{\lambda}}$. Since
%\begin{equation}
%\label{gj-C155}
$0 < \nu_{{\mu}^{\epsilon}} \cdot \nu_{{\lambda}} 
    \leq 1 - \beta^2/8$,
%\end{equation}
%Since $\nu_{\mu} \cdot \nu_{\lambda} \leq \beta < 1$ by
%Lemma~\ref{gj-C-lemma30},
$L$ intersects $\partial B_{\delta}^{\frac{\beta}{2}} ({\mu}^{\epsilon})$
at a unique point $\zeta$. By the concavity of $f$ we see that ,
\begin{equation}
\label{gj-C160}
\begin{aligned}
  \sum f_i ({\lambda}) ({\mu}^{\epsilon}_i - {\lambda}_i)
       = \,& \sum f_i ({\lambda}) (\zeta_i - {\lambda}_i) \\
  \geq \,& f (\zeta) - f ({\lambda}) \\
  \geq \,& \theta + f ({\mu}^{\epsilon}) - f ({\lambda}), \;\;
\forall \, 0 \leq \epsilon \leq \epsilon_0.
\end{aligned}
\end{equation}
%for any~$0\leq\epsilon\leq\epsilon_0$. 

Next, by the continuity of $f$ we may choose 
$0< \epsilon_1 \leq \epsilon_0$ with
$|f ({\mu}^{\epsilon_1})-f ({\mu})| \leq \frac{1}{2} \theta$. Hence
\begin{equation}
\sum f_{i} ({\lambda}) ({\mu}_i - \epsilon_1 - {\lambda}_i) 
 \geq f ({\mu}) - f({\lambda}) + \frac{1}{2} \theta.
\end{equation}
This proves \eqref{gj-C160'} with 
$\varepsilon = \min \{\theta/2, \epsilon_1\}$. 
\end{proof}

\begin{remark} 
Alternatively, one can first prove 
\[ \sum f_i ({\lambda}) ({\mu}_i - {\lambda}_i)
  \geq \theta + f ({\mu}) - f ({\lambda}). \] 
Then choose $\epsilon > 0$ small such that 
$0 \leq f (\mu) - f ({\mu}^{\epsilon}) \leq \frac{\theta}{2}$.
By the concavity of $f$, 
\begin{equation}
\begin{aligned}
  \sum f_i ({\lambda}) ({\mu}^{\epsilon}_i - {\lambda}_i)
  \geq \,& f (\mu^{\epsilon}) - f ({\lambda}) 
  \geq  f ({\mu}) - f ({\lambda}) - \frac{\theta}{2}.
\end{aligned}
\end{equation}
Now add these two inequlities to obtain \eqref{gj-C160'}.
\end{remark}

We now continue to prove \eqref{eq2-148}.
Assume first that $\ul u$ is a subsolution, i.e. $\ul u$ satisfies \eqref{eq1-11'}. 
Since $\lambda [\ul u]$ falls in a compact subset of $\Gamma$, 
\begin{equation}
\label{eq2-16}
\beta := \frac{1}{2} \min_{\bM} 
     \mbox{dist} (\nu_{\lambda [\ul u]}, \partial \Gamma_n) > 0. 
\end{equation}

%Finally, plug \eqref{gj-C160} and \eqref{gj-C165} into \eqref{hess-a276}

Let $\lambda = \lambda [u] (x_0, t_0)$ and  
$\mu = \lambda [\ul u] (x_0, t_0)$. 
If $|\nu_{\mu}- \nu_{\lambda}| \geq \beta$  
then by Lemma~\ref{ma-lemma-C10},
%\eqref{eq2-41} holds by
\begin{equation}
%\label{eq2-41}
F^{ii} \nabla_{ii} w - w_t
 \geq \sum f_i (\lambda) (\mu_i - \lambda_i) - f (\mu) - f (\lambda)
 \geq \varepsilon \Big(1 + \sum F^{ii}\Big).
\end{equation}
The first inequality follows from Lemma 6.2 in \cite{CNS3}; see \cite{Guan14}.
%\[ F^{ii} \nabla_{ii} w - w_t \geq \epsilon \Big(1 + \sum F^{ii}\Big). \]
 We may fix $a$ sufficiently large to derive a bound 
$U_{11} (x_0, t_0) \leq C$ by \eqref{eq2-141}.

Suppose now 
that $|\nu_{\mu}- \nu_{\lambda}| <  \beta$ and therefore
$\nu_{\lambda} - \beta {\bf 1} \in \Gamma_n$.
It follows that 
\begin{equation}
\label{eq2-22}
 F^{ii} \geq \frac{\beta}{\sqrt{n}} \sum F^{kk},
    \;\; \forall \, 1 \leq i \leq n.
\end{equation}
 Since $\ul{u}$ is a subsolution, 
$F^{ii} \nabla_{ii} w - w_t \geq 0$ by the concavity of $f$.
By \eqref{eq2-141} and \eqref{eq2-22} we obtain
\begin{equation}
\label{eq2-142}
 \begin{aligned}
\frac{b\beta }{2 \sqrt{n}} U_{11}^2 \sum F^{ii} + a \delta
      \leq C \sum F^{ii} + C.
\end{aligned}
\end{equation}

If we allow $\delta = 1$, a bound $U_{11} (x_0, t_0) \leq C$ would follow 
when $a$ is sufficiently large without using assumption~\eqref{eq1-0}.  
This gives \eqref{eq1-12a} in Theorem~\ref{thm2a}.

For the case $\delta = 0$, we need assumption~\eqref{eq1-0}. 
First, by the concavity of $f$,
\begin{equation}
\label{eq2-24a}
\begin{aligned}
  |\lambda| \sum f_i
  \geq \,& f (|\lambda| {\bf 1}) - f (\lambda)
        + \sum f_i \lambda_i \\
  \geq & f (|\lambda| {\bf 1}) - f (\lambda) 
        - \frac{1}{4 |\lambda|}
             \sum f_i \lambda_{i}^2 - |\lambda| \sum f_i.
    \end{aligned} 
\end{equation}
Hence, by assumption~\eqref{eq1-0},
\begin{equation}
\label{eq2-24}
\begin{aligned}
 U_{11}^2 \sum F^{ii} 
 \geq \,& \frac{U_{11}}{2 n} (f (U_{11} {\bf 1}) - u_t - \psi)
          - \frac{1}{8} \sum F^{ii} U_{ii}^2 \\
\geq \,& \frac{U_{11}}{2 n} - \frac{U_{11}^2}{8} \sum F^{ii}
    \end{aligned}
\end{equation}
 when $U_{11}$ is sufficiently large.
A bound $U_{11} (x_0, t_0) \leq C$ therefore follows from \eqref{eq2-142}.
The proof of \eqref{eq1-12} in Theorem~\ref{thm2} is complete.

\begin{remark}
\label{eq1-01a}
If \eqref{eq1-01} holds, a bound $U_{11} (x_0, t_0) \leq C$ follows from \eqref{eq2-142} directly (without using \eqref{eq1-0}) and is independent to $|u_t|_{C^0(\ol{M_T})}$. 
\end{remark}

Suppose now that $\ul u$ staisfies \eqref{eq1-11a} with the obvious
modification, i.e. with $\Sigma^{\sigma}$ redefined as
$\Sigma^{\sigma} = \{(\lambda, p) \in \Gamma \times \bfR:
f (\lambda) > p + \sigma\}$. 
By Lemma~\ref{ma-lemma-C20}  below
we have 
\begin{equation}
\label{eq2-41}
F^{ii} \nabla_{ii} w - w_t
% \geq \sum f_i (\lambda) (\mu_i - \lambda_i) - f (\mu) - f (\lambda)
 \geq \varepsilon \Big(1 + \sum F^{ii}\Big);
\end{equation}
when $U_{11}$ is sufficiently large.
Therefore, fixing $a$ large in \eqref{eq2-141} gives
$U_{11} (x_0, t_0) \leq C$. 
This completes the proof of \eqref{eq1-12} in Theorem~\ref{thm3},
subject to the proof of Lemma~\ref{ma-lemma-C20}.

\begin{lemma}
\label{ma-lemma-C20} 
Let $K$ be a compact subset of $\Gamma \times \bfR$ such that 
$\hat S^{\sigma}_{\hat{\mu}} [a, b] 
  := \hat S^{\sigma}_{\hat{\mu}} \cap \Gamma \times [a, b]$
is compact for any $\hat \mu \in K$. Then there exist $R, \epsilon > 0$ such that for all $\lambda \in \Gamma^{[a, b]}$, 
when $|\lambda| > R$,
\begin{equation}
\label{gj-C160''}
  \sum f_i (\lambda) (\mu_i - \lambda_i) \geq z - f(\lambda) 
         + \varepsilon \Big( 1 + \sum f_i (\lambda)\Big),
\;\; \forall \; (\mu, z) \in K.
\end{equation}
\end{lemma}

In some sense this is a parabolic version of Theorem 2.17 in \cite{Guan14}. Its proof is long but follows similar ideas in  \cite{Guan14}. So we include it in the Appendix for completeness 
and for the reader's convenience.

\begin{remark}
\label{eq2-199}
If $\ul u$ is an admissible {strict} subsolution, i.e. 
\begin{equation}
\label{eq1-11ss}
%\left\{
\begin{aligned}
%\frac{\partial \ul u}{\partial t}  
  f(\lambda [\ul u]) \geq  {\ul u_t + \psi} + \delta
         %\nabla^2\ul{u} + \chi)) - \psi
                                              \;\; \mbox{in $M_T$}  \\
 % \ul u \leq \,& \varphi  \;\; \mbox{on $\partial_b M_T$}, \\
 % \ul u = \,& \varphi  \;\; \mbox{on $\partial_s M_T$} .
 \end{aligned} %\right.
\end{equation}
for some $\delta > 0$, then we can choose $\epsilon > 0$ such that 
$\lambda^{\epsilon} [\ul u] :=\lambda [\ul u] - \epsilon {\bf 1} \in \Gamma$ and 
\begin{equation}
\label{eq1-11ss'}
%\left\{
\begin{aligned}
%\frac{\partial \ul u}{\partial t}  
  f (\lambda^{\epsilon} [\ul u]) \geq  {\ul u_t + \psi} + \frac{\delta}{2}
         %\nabla^2\ul{u} + \chi)) - \psi
                                              \;\; \mbox{in $M_T$} . \\
 % \ul u \leq \,& \varphi  \;\; \mbox{on $\partial_b M_T$}, \\
 % \ul u = \,& \varphi  \;\; \mbox{on $\partial_s M_T$} .
 \end{aligned} %\right.
\end{equation}
By the concavity of $f$ we see that 
\[ \sum f_i (\lambda [u]) (\lambda^{\epsilon}_i [\ul u] - \lambda_i [u])
   \geq  f (\lambda^{\epsilon} [\ul u]) - f (\lambda [u])
   \geq  {\ul u}_t - u_t + \frac{\delta}{2}. \]
Therefore one can derive \eqref{eq2-148} directly from 
Proposition~\ref{eq2-149}. This can be used to prove Theorem~\ref{thm2a} as $\ul u = \varphi^b + A t$ is a 
strict subsolution of equation \eqref{eq1-1a} for any constant 
$A < \inf_{M}  f (\lambda [\varphi^b]) - \sup_{M_T} \psi$. 
\end{remark}

\bigskip

\section{Second order boundary estimates}
\label{3I-B}
\setcounter{equation}{0}

\medskip

Let $u \in C^{3, 1} (\ol{M_T})$ be an admissible solution of 
\eqref{eq1-1a} satisfying %the parabolic boundary condition~
\eqref{eq1-1IB} and the $C^1$ estimate \eqref{eq2-195}.
In this section we derive \eqref{eq1-14} under the assumptions
\eqref{eq1-3}, \eqref{eq1-4}, \eqref{eq1-0}, \eqref{eq1-13}
and \eqref{eq1-9} on $f$. Clearly we only need to focus
on $\partial_s M_T$. 

For a point $x_0\in \partial M$ we shall choose smooth orthonormal 
local frames $e_1, \ldots, e_n$ around $x_0$ such that 
$e_n$, when restricted to $\partial M$,  is the interior unit normal to $\partial M$. 
By the boundary condition
$u = \varphi^s$ on $\partial_s M_T$ we obtain 
\begin{equation}
|\nabla_{\alpha\beta} u (x_0, t_0) |\leq C, \;\; \forall \, 1\leq\alpha,
\beta<n, \;\; \forall \, 0 \leq t \leq T.
 \end{equation}

Let $\rho (x)$ and  $d (x)$ denote the distances from $x \in \bM$ to
$x_0$ and $\partial M$, respectively. 
Let $M_T^\delta = \{(x, t) \in M_T: \rho(x) < \delta\}$, 
and $\partial M_T^\delta$ be the parabolic boundary of $M_T^\delta$,
\[ \partial M_T^\delta = \ol{M_T^\delta} \setminus M_T^\delta. \]
%\begin{equation*}
%\mathcal{P}M_\delta = \{ (x, 0)\in BM: \rho(x) < \delta \}\cup
%(SM_T\cap \ol{M_\delta}) \cup \{(x , t): \rho(x)=\delta, t \leq t_0 
%+ \delta\}
%\end{equation*}
We fix $\delta_0 > 0$ sufficiently small such that both $\rho$ and $d$
are smooth in $M_T^{\delta_0}$. Let $\mathcal{L}$ denote the linear
parabolic operator 
\[ \mathcal{L} w = F^{ij} \nabla_{ij} w - w_t, \]
and 
\begin{equation}
\label{eq3-3}
 \varPsi
   = A_1 v + A_2 \rho^2 - A_3 \sum_{l< n} |\nabla_l (u - \varphi)|^2
             % \pm \nabla_{\alpha} (u - \varphi)
\end{equation}
where
\begin{equation}
\label{eq3-4} 
v = u - \ul{u} + s d - \frac{N d^2}{2}
\end{equation}
and $\ul u \in C^{2, 1} (\ol{M_T})$ is an admissible 
subsolution satisfying \eqref{eq1-11'} and \eqref{eq1-11b}.

\begin{lemma}
\label{ma-lemma-E10} 
Assume \eqref{eq1-3}, \eqref{eq1-4}, \eqref{eq1-0}, 
\eqref{eq1-13} hold
and $\ul u$ satisfies \eqref{eq1-11'} and \eqref{eq1-11b}.
Then for constant $K > 0$, there exist uniform positive 
constants $s$, $\delta$ sufficiently small, 
and $A_1$, $A_2$, $A_3$, $N$
sufficiently large such that 
$\varPsi \geq K (d + \rho^2)$ in $\ol{M_T^{\delta}}$ and 
\begin{equation}
\begin{aligned}
\label{eq3-5} 
\mathcal{L} \varPsi 
   \leq \,& - K \Big(1 + \sum f_i |\lambda_i|
                 + \sum f_i\Big)  \;\; \mbox{in $M_T^{\delta}$}.
\end{aligned}
\end{equation}
\end{lemma}

\begin{proof} 
This is parabolic version of Lemma~3.1 in \cite{Guan}. 
Since there are some substantial differences in several places,
for completeness and reader's convenience we include 
a detailed proof.  

First we note that $\mathcal{L}  (u - \ul{u}) \leq 0$
by the concavity of $f$ and since $\ul{u}$ is a subsolution, 
and by \eqref{eq2-4} , 
\begin{equation}
\label{eq3-1}
\begin{aligned}
|\mathcal{L} \nabla_k (u - \varphi)| 
    \leq \,& C \Big(1 + \sum f_i |\lambda_i| + \sum
f_i\Big),
 \;\; \forall \; 1 \leq k \leq n.
\end{aligned}
\end{equation}
It follows that
\begin{equation}
\label{eq3-6}
\begin{aligned}
 \sum_{l < n} \mathcal{L} |\nabla_l (u - \varphi)|^2 
 \geq \, \sum_{l < n} F^{ij} U_{i l} U_{j l}
         - C \Big(1 + \sum f_i |\lambda_i| + \sum f_i\Big).
\end{aligned}
\end{equation}
By Proposition~2.19 in \cite{Guan14} there exists an index $r$ such
that
\begin{equation}
\label{eq3-7} \sum_{l < n} F^{ij} U_{i l} U_{j l}
     \geq \frac{1}{2} \sum_{i \neq r} f_i \lambda_i^2.
\end{equation}

At a fixed point $(x, t)$, denote
${\mu} = \lambda (\nabla^2 \ul u +  \chi)$
and $\lambda = \lambda (\nabla^2 u + \chi)$.
As in Section~\ref{3I-C} 
we consider two cases separately:
 {\bf (a)} $|\nu_{\mu}- \nu_{\lambda}| < \beta$
 and
{\bf (b)} $|\nu_{\mu}- \nu_{\lambda}| \geq \beta$, 
where $\beta$ is given  in \eqref{eq2-16}.

Case {\bf (a)} $|\nu_{\mu}- \nu_{\lambda}| < \beta$.
We have by \eqref{eq2-22}
\begin{equation}
\label{eq3-9}
 f_i \geq \frac{\beta}{\sqrt{n}} \sum f_k,
    \;\; \forall \, 1 \leq i \leq n.
\end{equation}
We next show that this implies the following inequality for any index $r$
\begin{equation}
\label{eq3-8}
\sum_{i \neq r} f_i \lambda_i^2 %\sum_{l < n} F^{ij} U_{i l} U_{j l}
     \geq c_0 \sum f_i \lambda_i^2 - C_0 \sum f_i
\end{equation}
for some $c_0, C_0 > 0$. 

Since $\sum \lambda_i \geq 0$, we see that 
\begin{equation}
\label{eq3-10}
 \sum_{\lambda_i < 0} \lambda_i^2
     \leq \Big(- \sum_{\lambda_i < 0} \lambda_i\Big)^2
    \leq n \sum_{\lambda_i > 0} \lambda_i^2.
\end{equation}
Therefore, by \eqref{eq3-9} and \eqref{eq3-10} we obtain 
if $\lambda_r < 0$, 
\[ f_r \lambda_r^2 \leq n f_r \sum_{\lambda_i > 0} \lambda_i^2
    \leq \frac{n \sqrt{n}}{\beta} \sum_{\lambda_i > 0} f_i \lambda_i^2. \]

On the other hand, by the concavity of $f$ and assumption \eqref{eq1-0}
we have  
\begin{equation}
\label{eq3-101}
\sum f_i (b - \lambda_i) \geq f (b {\bf 1}) - f (\lambda)
= f (b {\bf 1}) - u_t - \psi \geq 1 
\end{equation}
for $b > 0$ sufficiently large. 
It follows that if $\lambda_r > 0$, 
\[ f_r \lambda_r \leq b \sum f_i - \sum_{\lambda_i < 0} f_i \lambda_i. \]
By \eqref{eq3-9} and Schwarz inequality, 
\[  \begin{aligned}
  \frac{\beta f_r \lambda_r^2}{\sqrt{n}} \sum f_k
  \leq \,& f_r^2 \lambda_r^2
  \leq   2 b^2 \Big(\sum f_i\Big)^2
        + 2 \sum_{\lambda_k < 0} f_k  \sum_{\lambda_i < 0} f_i \lambda_i^2 \\
  \leq \,& 2 \Big(\sum_{\lambda_i < 0} f_i \lambda_i^2
              + b^2 \sum f_i \Big) \sum f_k.
\end{aligned}    \]
This finishes the proof of \eqref{eq3-8}.

Letting  $b = n |\lambda|$ in \eqref{eq3-101}, we see that 
\begin{equation} 
\label{eq3-11'}
(n+1) |\lambda| \sum f_i \geq 
\sum f_i (n |\lambda| - \lambda_i)
\geq f (n |\lambda| {\bf 1}) - f (\lambda) \geq 1,
\end{equation}
and consequentley by \eqref{eq3-9},
\begin{equation} \label{eq3-11}
\sum f_i \lambda_i^2 \geq \frac{\beta |\lambda|^2}{\sqrt{n}} \sum f_i 
                    \geq \frac{\beta |\lambda|}{(n+1)\sqrt{n}},
\end{equation}
provided that $|\lambda| \geq R$ for $R$ sufficiently large.

It now follows from 
\eqref{eq3-6}, \eqref{eq3-7}, \eqref{eq3-8}, \eqref{eq3-11} and 
Schwartz inequality that when $|\lambda|\geq R$, 
\begin{equation}
\label{eq3-12}
\begin{aligned}
\sum_{l < n} \mathcal{L} |\nabla_l (u - \varphi)|^2 
\geq \,& c_1 \sum f_i \lambda_i^2 + 2 c_1 |\lambda| - C - C_1 \sum f_i.
\end{aligned}
\end{equation}
for some $c_1, C_1 > 0$. We now fix $R \geq C/c_1$. 

Turning to the fucntion $v$, we note that by \eqref{eq3-9},
\begin{equation}
\label{eq3-13}
\begin{aligned}
\mathcal{L} v  \leq \,& \mathcal{L} (u- \ul u) + C (s+N d) \sum F^{ii} 
        - N  F^{ij} \nabla_{i} d \nabla_{j} d \\
\leq \,&  \Big(C (s+N d) - \frac{\beta N}{\sqrt{n}}\Big) \sum F^{ii}  
\end{aligned}
\end{equation}
since $\mathcal{L}  (u - \ul{u}) \leq 0$ and $|\nabla d| \equiv 1$.
%by the concavity of $f$ and the fact that $\ul{u}$ is a subsolution. 
For $N$ sufficiently large we have
\begin{equation}
\label{eq3-14} \mathcal{L} v \leq - \sum f_i
   \;\; \mbox{in $M_T^{\delta}$}
\end{equation}
and therefore, in view of \eqref{eq3-12} and \eqref{eq3-14}, 
\begin{equation}
\label{eq3-15}
\begin{aligned}
\mathcal{L} \varPsi \leq \,&
       - A_3 c_1 \Big(|\lambda| + \sum f_i \lambda_i^2\Big)
     %- A_3 K \Big(1 + \sum f_i |\lambda_i|\Big)
        + ( - A_1 + C A_2 + C_1 A_3) \sum f_i. 
%\geq \,& F^{ij} U_{i \beta} U_{j \beta} - \epsilon \sum f_i \lambda_i^2
%         - C \Big(1 + \sum f_i\Big)
\end{aligned}
\end{equation}
when $|\lambda| \geq R$, for any $s \in (0,1]$ as long as $\delta$ is 
sufficiently small. 
From now on $A_3$ is fixed such that $A_3 c_1 R \geq K$, so 
$A_3 \geq C K/c_1^2$. 

Suppose now that $|\lambda| \leq R$. 
By \eqref{eq1-3} and \eqref{eq1-4} we have 
%the concavity of $f$ we have
\begin{equation}
\label{eq3-155}
 \begin{aligned}
 2R \sum f_i 
    \geq \,& \sum f_i \lambda_i + f (2 R {\bf 1}) - f (\lambda) \\
    \geq \,& - R \sum f_i + f (2 R {\bf 1}) - f (R {\bf 1}). 
\end{aligned}
\end{equation}

Therefore,
\[ \sum f_i \geq \frac{f ( 2 R {\bf 1}) - f (R {\bf 1})}{3R} 
     \equiv C_R > 0. \]
It follows from \eqref{eq2-22} that there is a uniform lower bound
\begin{equation}
\label{eq3-16}
  f_i \geq \frac{\beta}{\sqrt{n}} \sum f_k \geq  \frac{\beta C_R}{\sqrt{n}},   \;\; \forall \, 1 \leq i \leq n.
\end{equation}
Consequently, since $|\nabla d| = 1$, 
\[ F^{ij} \nabla_{i} d \nabla_{j} d \geq \frac{\beta}{2 \sqrt{n}}
\Big(C_R + \sum f_i\Big). \]
From \eqref{eq3-13} we see that when $\delta$ is sufficiently small 
and $N$ sufficiently large,
\begin{equation}
\label{eq3-17} 
\mathcal{L} v \leq - \Big(1 + \sum f_i\Big)
   \;\; \mbox{in $M_T^{\delta}$}.
\end{equation}
Combining \eqref{eq3-6}, \eqref{eq3-7}, \eqref{eq3-8}, \eqref{eq3-17}
yields
\begin{equation}
\label{eq3-15'}
\begin{aligned}
\mathcal{L} \varPsi \leq \,&
       - A_3 c_1 \sum f_i \lambda_i^2
     %- A_3 K \Big(1 + \sum f_i |\lambda_i|\Big)
        + (- A_1 + C A_2 + C A_3) \sum f_i - A_1 + C A_3
%\geq \,& F^{ij} U_{i \beta} U_{j \beta} - \epsilon \sum f_i \lambda_i^2
%         - C \Big(1 + \sum f_i\Big)
\end{aligned}
\end{equation}
We now fix $N$ such that \eqref{eq3-14} holds when $|\lambda| > R$
while \eqref{eq3-17} holds when $|\lambda| \leq R$, for any $s$
and $\delta$ sufficiently small.

Case {\bf (b)} $|\nu_{\mu}- \nu_{\lambda}| \geq \beta$.
It follows from Lemma~\ref{ma-lemma-C10} that, for some $\varepsilon > 0$, 
\[ \mathcal{L} (\ul u - u)  
    \geq \sum f_i (\mu_i - \lambda_i) - (\ul u - u)_t
    \geq \varepsilon \Big(1 + \sum f_i\Big) \]
By \eqref{eq3-13}, we may fix $s$ and $\delta$ suffieicently 
small such that $v \geq 0$ on $\ol{M_T^\delta}$
and 
\begin{equation}
\label{eq3-18} 
\mathcal{L} v \leq - \frac{\varepsilon}{2} \Big(1 + \sum f_i\Big)
   \;\; \mbox{in $M_T^{\delta}$}.
\end{equation}

Finally, we choose $A_2$ large such that 
\[ (A_2 - K) \rho^2 \geq A_3 \sum_{l< n} |\nabla_l (u - \varphi)|^2
    \;\; \mbox{on $\partial M_T^{\delta}$}, \]
and then fix $A_1$ sufficiently large so that \eqref{eq3-5} holds; 
in case {\bf (a)} this follows
from \eqref{eq3-15} when $|\lambda| > R$, and from
\eqref{eq3-15'} % \eqref{eq3-7}, \eqref{eq3-16} and \eqref{eq3-17} 
when $|\lambda| \leq R$, while  
in case {\bf (b)} we obtain \eqref{eq3-5} from
\eqref{eq3-6}, \eqref{eq3-7}, \eqref{eq3-18} and the following
inequality 
\begin{equation}
\label{eq3-19}
 \sum f_i |\lambda_i| \leq \epsilon \sum_{i \neq r} f_i \lambda_i^2
  + \frac{C}{\epsilon} \sum f_i + C
\end{equation}
for any $\epsilon > 0$ and index $r$, which is a consequence of
\eqref{eq1-3}, \eqref{eq1-4} %\eqref{eq1-0} 
and \eqref{eq1-13}. 
For the proof of \eqref{eq3-19}, we consider two cases. 
If $\lambda_r < 0$ then, by \eqref{eq1-13}
\[ \begin{aligned}
   \sum f_i |\lambda_i|
   = \,& 2\sum_{\lambda_i > 0} f_i \lambda_i - \sum  f_i \lambda_i \\
\leq  \,& \epsilon \sum_{\lambda_i > 0} f_i \lambda_i^2 +
\frac{1}{\epsilon} \sum_{\lambda_i > 0} f_i + K_1 \sum f_i + K_1. 
  \end{aligned} \]
 If $\lambda_r > 0$, we have by the
concavity of $f$, 
\[ \begin{aligned}
   \sum f_i |\lambda_i|
    = \,& \sum  f_i \lambda_i - 2\sum_{\lambda_i < 0} f_i \lambda_i \\
\leq \,& \epsilon \sum_{\lambda_i < 0} f_i \lambda_i^2
            + \frac{1}{\epsilon} \sum_{\lambda_i < 0} f_i + \sum f_i 
            + f (\lambda) - f ({\bf 1}).
     \end{aligned} \]
This proves \eqref{eq3-19}.
\end{proof}

Applying Lemma~\ref{ma-lemma-E10}, by \eqref{eq3-1} we 
immediately 
derive a bound for the mixed tangential-normal
derivatives at any point $(x_0, t_0) \in \partial M_T$,
\begin{equation}
\label{eq3-20} |\nabla_{n\alpha} u (x_0, t_0)| \leq C, \;\; \forall
\; \alpha < n
\end{equation}
It remains to establish the double normal derivative estimate
\begin{equation}
\label{eq3-21}
 |\nabla_{n n} u (x_0, t_0)| \leq C.
\end{equation}
As in \cite{Guan14} and \cite{Guan} we use an idea originally due to
Trudinger \cite{Trudinger95}.

For $(x, t) \in \partial_s M_T$, let $\tilde{U} (x, t)$ be the
restriction to $T_x \partial M$ of $U (x, t)$, viewed as a bilinear map 
on the tangent space %$T_x M$, 
of $M$ at $x$, 
and let $\lambda' (\tilde{U})$ denote the eigenvalues of
$\tilde{U}$ with respect to the induced metric on $\partial M$. 
We show next that there are uniform positive constants $c_0, R_0$ 
such that, for all $R > R_0$,
$(\lambda' (\tilde{U} (x, t)), R) \in \Gamma$ and
\begin{equation}
\label{eq3-22} 
f (\lambda' (\tilde{U} (x, t)), R) \geq
f(\lambda(U(x, t))) + c_0,
\;\; \forall  \; 0 \leq t \leq T, \; \forall \; x \in \partial M.
\end{equation}
It is known that \eqref{eq3-22} implies \eqref{eq3-21}; see e.g. \cite{Guan14}.

For $R > 0$ sufficiently large, let 
\[  \begin{aligned} 
  m_R := \,& \min_{\partial_s M_T} 
                    [f (\lambda' (\tilde{U}), R) -  f (\lambda (U)],  \\
   c_R := \,& \min_{\partial_s M_T}
                   [f (\lambda' (\tilde{\ul U}), R) - f (\lambda{(\ul U})]. 
    \end{aligned} \]
Note that $(\lambda' (\tilde{U} (x, t)), R) \in \Gamma$ and
$(\lambda' (\tilde{U} (x, t)), R) \in \Gamma$ for all 
$(x, t) \in \partial_s M_T$ for all $R$ large, and it is clear that 
both $m_R$ and $c_R$ are increasing in $R$. 
We wish to show that for some uniform $c_0 > 0$,  
\[  \tilde{m} := \lim_{R \rightarrow \infty} m_R \geq c_0. \]

Assume $\tilde{m} < \infty$ (otherwise we are done) and fix $R > 0$  
such that $c_R > 0$ and $m_R \geq \frac{\tilde{m}}{2}$. 
Let $(x_0, t_0) \in \partial_s M_T$ such that 
$m_R = f (\lambda' (\tilde{U} (x_0, t_0)), R)$.
Choose local orthonormal frames $e_1, \ldots, e_n$ around $x_0$ as before such 
that $e_n$ is the interior normal to $\partial M$ along the boundary and 
$U_{\alpha \beta} (x_0, t_0)$ ($1 \leq \alpha, \beta \leq n-1$) is diagonal. 
Since $u - \ul u = 0$ on $\partial_s M_T$, we have 
\begin{equation}
\label{eq3-24}
 U_{\alpha {\beta}} - \ul{U}_{\alpha {\beta}}
    = - \nabla_n (u - \ul{u}) \sigma_{\alpha {\beta}}
\;\; \mbox{on $\partial_s M_T$}.
\end{equation}
where $\sigma_{\alpha {\beta}} = \langle \nabla_{\alpha} e_{\beta}, e_n \rangle$. 
Similarly, 
\begin{equation}
\label{eq3-24'}
 U_{\alpha {\beta}} - \nabla_{\alpha {\beta}} \varphi - \chi_{\alpha {\beta}} \varphi
    = - \nabla_n (u - \varphi) \sigma_{\alpha {\beta}}
\;\; \mbox{on $\partial_s M_T$}.
\end{equation}

For an $(n-1) \times (n-1)$ symmetric matrix $\{r_{\alpha, \beta}\}$ with 
$(\lambda' (\{r_{\alpha, \beta}\}), R) \in \Gamma$, define
\[ \tF [r_{\alpha \beta}] := f (\lambda' (\{r_{\alpha, \beta}\}), R) \]
and  
\[ \tF^{\alpha\beta}_0 
  = \frac{\partial \tF}{\partial r_{\alpha\beta}}[U_{\alpha\beta}( x_0, t_0 )]. \] 
We see that $\tF$ is concave since so is $f$, and therefore by
\eqref{eq3-24},  
\[   \nabla_n (u - \ul{u}) (x_0, t_0) \tF^{\alpha {\beta}}_0 
               \sigma_{\alpha {\beta}} (x_0)
      \geq \tF[\ul{U}_{\alpha {\beta}} (x_0, t_0)] 
            - \tF[U_{\alpha {\beta}} (x_0, t_0)]
 \geq c_R - m_R. \]

Suppose that
\[ \nabla_n (u - \ul{u}) (x_0, t_0)
    \tF^{\alpha {\beta}}_0 \sigma_{\alpha {\beta}} (x_0)
   \leq \frac{c_R}{2} \]
then $m_R \geq c_R/2$ and we are done.
So we shall assume
\[ \nabla_n (u - \ul{u}) (x_0, t_0) 
   \tF^{\alpha {\beta}}_0 \sigma_{\alpha {\beta}} (x_0) > \frac{c_R}{2}. \]
Consequently, 
\begin{equation}
\label{eq3-25} 
\tF^{\alpha\beta}_0 \sigma_{\alpha\beta}(x_0) \geq
\frac{c_R}{2 \nabla_n (u - \ul{u}) (x_0, t_0)} \geq 2 \epsilon_1 c_R
\end{equation}
for some constant $\epsilon_1 > 0$ depending on 
$\max_{\partial_s M_T} |\nabla u|$.   
By continuity we may assume 
$\eta := \tF^{\alpha\beta}_0 \sigma_{\alpha\beta} \geq \epsilon_1 c_R$ 
on $\ol{M_T^{\delta}}$ by requiring $\delta$ small (which may 
depend on the fixed $R$).
Define in $M_T^{\delta}$, 
\begin{equation}
\begin{aligned}
  \varPhi = \,& - \nabla_n (u - \varphi) + \frac{Q}{\eta}
 %  \varPhi = \,& - \nabla_n (u - \ul u) + \frac{Q}{\eta}
\end{aligned}
\end{equation}
where 
\[ Q = \tF^{\alpha {\beta}}_0 %\ul U_{\alpha\beta} 
   (\nabla_{\alpha\beta} \varphi + \chi_{\alpha\beta} 
        - U_{\alpha\beta} (x_0, t_0))
        - \ul{u}_t - \psi + u_t (x_0, t_0) + \psi (x_0, t_0) \]
is smooth in $M_T^{\delta}$.
By \eqref{eq3-1} we have 
\begin{equation}
\label{eq3-26}
 \begin{aligned}
\mathcal{L} \varPhi 
  \leq \, & - \mathcal{L} \nabla_n u + C \Big(1 + \sum F^{ii}\Big) 
  \leq   C \Big(1 + \sum f_i |\lambda_i| + \sum f_i\Big).
\end{aligned}
\end{equation}

From \eqref{eq3-24'} we see that $\varPhi (x_0, t_0) = 0$ and 
\begin{equation}
\label{eq3-25a} 
\varPhi \geq 0 \;\; \mbox{on $\ol{M_T^\delta} \cap \partial_s M_T$},
\end{equation}
 since for $(x, t) \in \partial_s M_T$, by the concavity of $\tF$,
\[ \begin{aligned}
\tF_0^{\alpha\beta}( U_{\alpha\beta}(x, t) - U_{\alpha\beta}(x_0, t_0)) 
  \geq \,&  \tF (\tU (x, t)) - \tF (\tU (x_0, t_0)) \\
      = \,& \tF (\tU (x, t)) - m_R - u_t (x_0, t_0) - \psi (x_0, t_0) \\
\geq \,& \psi (x, t) + u_t (x, t)  - u_t (x_0, t_0) - \psi (x_0, t_0).
\end{aligned} \]
On the other hand, 
on $\partial_b M_T^{\delta}$ we have %$\varPhi \geq - C \rho^2$ 
$\nabla_n (u - \varphi) = 0$ and therefore, by \eqref{eq3-25a}, 
%by the concavity of $\tF$,
\begin{equation}
\label{eq3-25b} 
\begin{aligned}
\varPhi (x, 0) \geq \varPhi (\hat{x}, 0) - C d (x) \geq - C d (x), 
%\,& \tF^{\alpha {\beta}}_0 
 %   (U_{\alpha\beta} (\hat{x}, 0)- U_{\alpha\beta} (x_0, t_0))
%       - \ul{u}_t (\hat{x}, 0) \\
%   & - \psi (\hat{x}, 0) + u_t (x_0, t_0) + \psi (x_0, t_0) - C d (x) \\
% \geq \,& \tF (\tU (\hat{x}, 0)) - \tF (\tU (x_0, t_0)) 
 %       - u_t (\hat{x}, 0) \\
 %  & - \psi (\hat{x}, 0) + u_t (x_0, t_0) + \psi (x_0, t_0) - C d (x)\\
 %= \,& \tF (\tU (\hat{x}, 0)) - u_t (\hat{x}, 0) - \psi (\hat{x}, 0) 
 %       - m_R - C d (x) \\ \geq \,&  - C d(x)
  \end{aligned} 
\end{equation}
%since $Q (\hat{x}, 0) \geq 0$ 
where $C$ depends on $C^1$ bounds of 
$\nabla^2 \varphi (\cdot, 0)$, %|_{C^1 (\partial_b M_T)}$,  
$\ul u_t  (\cdot, 0)$, $\psi  (\cdot, 0)$ on $\bM$, and 
$\hat{x} \in \partial M$ satisfies $d (x) = \mbox{dist} (x, \hat{x})$
for $x \in M$; when $d (x)$ is sufficiently small, $\hat{x}$ is unique. 

Finally, note that $|\varPhi| \leq C$ in $M_T^{\delta}$. 
So we may apply Lemma~\ref{ma-lemma-E10} to derive
$\varPsi + \varPhi \geq 0$ on $\partial M_T^{\delta}$ and
\begin{equation}
\label{eq3-27} 
\mathcal{L} (\varPsi + \varPhi) \leq 0 \;\; \mbox{in $M_T^{\delta}$}
\end{equation}
for $A_1$, $ A_2$, $A_3$ sufficiently large. 
By the maximum principle, $\varPsi +  \varPhi \geq 0$ in
$M_T^{\delta}$. This gives
$\nabla_n \varPhi (x_0, t_0) \geq - \nabla_n \varPsi (x_0, t_0) \geq -C$
since $\varPhi + \varPsi = 0$ at $(x_0, t_0)$, and therefore, 
%\begin{equation}
%\label{cma-310}
$\nabla_{nn} u (x_0, t_0) \leq C$. %\frac{C}{\eta (x_0)} \leq \frac{C}{c_R}.
%\end{equation}

Consequently, we have obtained {\em a priori} bounds for all 
second derivatives of $u$ at $(x_0, t_0)$. 
It follows that $\lambda (U (x_0, t_0))$ is contained in a 
compact subset (independent of $u$) of $\Gamma$ by
assumptions \eqref{eq1-5}. Therefore,
\[ c_0 \equiv \frac{f  (\lambda (U (x_0, t_0)) + R \ve_n)  
        - f (\lambda (U (x_0, t_0)))}{2} > 0 \]
where $\ve_n = (0, \dots, 0, 1) \in \bfR^n$.
By Lemma~1.2 in \cite{CNS3} we have 
\[ \begin{aligned}
    \tm \geq \,& m_{R'}
  %  f (\lambda' (\tU (x_0, t_0)), R') - f (\lambda (U (x_0, t_0))) \\
\geq f  (\lambda (U (x_0, t_0)) + R' \ve_n) - c_0 - f (\lambda (U (x_0, t_0))) 
     \geq c_0  
\end{aligned} \]
for $R' \geq R$ sufficiently large.
 The proof of \eqref{eq1-14} in Theorem~\ref{thm2} is complete.

\begin{remark}
\label{remak-3.10}
When $M$ is a bounded smooth in $\bfR^n$, one can make use of 
an identity in \cite{CNS3}, and modify the operator $\mathcal L$,
to derive the boundary estimates without 
using assumption~\eqref{eq1-13}. 
We omit the proof here since it is similar to the elliptic case in \cite{Guan} 
which we refer the reader to for details.
\end{remark}

\bigskip

\section{Existence and $C^1$ estimates}
\label{E}
\setcounter{equation}{0}

\medskip

In order to prove Theorem~\ref{thm1} it remains to derive 
the $C^1$ estimate
\begin{equation}
%\label{eq4-1} %\eqref{eq2-195}
|u|_{C^0 (\ol{M_T})} 
+ \max_{\bM \times [t_0, T]} (|\nabla u| + |u_t|) \leq C
 % \;\; \forall \, t_0 \in (0, T)
\end{equation} 
 for any $t_0 \in (0, T)$, where $C$ may depend on $t_0$.
Indeed, by assumption \eqref{eq1-5} we see that equation~\eqref{eq1-1} 
becomes uniformly parabolic once the $C^{2,1}$ estimate
\[ |u|_{C^{2,1} (\bM \times [t_0, T])} \leq C \]
is established, which yields 
$|u|_{C^{2+\alpha,1+\alpha/2} (\bM \times [t_0, T])} \leq C$
by Evans-Krylov theorem~\cite{Evans82, Krylov} (see e.g. \cite{Lieberman}).
Higher order estimates now
follow from the classical Schauder theory of linear parabolic
equations, and one obtains a smooth admissible solution in 
$0 \leq t \leq T$ by the short time existence and continuation. 
We refer the reader to \cite{Lieberman} for details.

Let $h \in C^2 (\bM_T)$ be the solution of 
$\Delta h + \tr \chi = 0$ in $\bM_T$ 
with $h = \varphi$ on $\partial M_T$. 
By the maximum principle we have $\ul u \leq u \leq h$ which gives
a bound
\begin{equation}
%\label{eq4-2}
|u|_{C^0 (\ol{M_T})} + \max_{\partial M_T} |\nabla u| \leq C. 
\end{equation} 

For the bound of $u_t$ we have the following maximum principle.

\begin{lemma}
\label{lemma-E10}
\begin{equation}
\label{eq2-a}
 |u_t (x, t)| \leq \max_{\partial M_T} |u_t| + t \sup_{M_T} |\psi_t|, 
     \;\; \forall \, (x, t) \in \bM_T
\end{equation}
Moreover, if there is a convex function in $C^2 (\bM)$ 
%(e.g. when $M$ is a bounded domain in $\bfR^n$) 
then 
\begin{equation}
\label{eq2-ab}
 \sup_{M_T} |u_t| \leq \max_{\partial M_T} |u_t | 
    + C \sup_{M_T} |\nabla^2 \psi|
\end{equation}
where $C$ is independent of $T$.
\end{lemma}

\begin{proof}
We have the following identities:
%Let $\mathcal{L}$ be the linear opertor 
%locally given by $\mathcal{L} v = F^{ij} \nabla_{ij} v - v_t$.
$\mathcal{L} u_t = \psi_t$
and 
\[  |\mathcal{L} (u_t+\psi)| = | F^{ij} \nabla_{ij} \psi|
        \leq  |\nabla^2 \psi| \sum F^{ii}. \]
So Lemma~\ref{lemma-E10} is an immediate 
consequence of the maximum principle.
\end{proof}

It remains to derive the gradient estimate
\begin{equation}
\label{gradient}
\sup_{M_T}|\nabla u|^2 \leq C \Big(|u|_{C^0(\ol{M_T})}
%\big(\sup\limits_{\mathcal{P}M_T} u -\inf\limits_{M_T} u \big )
+ \sup\limits_{\partial M_T} |\nabla u|^2\Big)
\end{equation} 
 in each of the 
cases (i)-(iv) in Theorem~\ref{thm1}. 
We shall omit case (i) which is trivial, and consider cases (ii)-(iv) following 
ideas from \cite{ LiYY90, Urbas02, Guan14}  in the elliptic case.

%Case ({\sl ii}) and ({\sl iv}): 
Let $\phi$ be a function to be chosen and assume that $|\nabla u| e^{\phi}$
%$|\nabla u|^2 e^{A (\ul{u} - u)}$
achieves a maximum at an interior point $(x_0, t_0) \in M_T$. As
before we choose local orthonormal frames at $x_0$ such that both
$U_{ij} $ and $F^{ij}$ are diagonal at $(x_0, t_0)$ where 
\begin{equation} 
\label{eq4-0-1}
\frac{\nabla_k u\nabla_k u_t}{|\nabla u|^2} + \phi_t \geq 0, \;\;  \frac{\nabla_k u \nabla_{ik} u}{|\nabla u|^2} + \nabla_i \phi = 0,
\;\; \forall \,  i = 1, \cdots, n,
\end{equation}
\begin{equation} 
\label{eq4-0-3}
F^{ii} \frac{\nabla_k u \nabla_{iik} u + \nabla_{ik} u \nabla_{ik}
u}{|\nabla u|^2} - 2 F^{ii} \frac{(\nabla_k u \nabla_{ik}
u)^2}{|\nabla u|^4}
   %  \nabla_l u \nabla_{il} u + |\nabla u|^2
+ F^{ii} \nabla_{ii} \phi
 %  + |\nabla u|^2 F^{ij} (\nabla_{ij} \phi - \nabla_i \phi \nabla_j \phi)
  \leq 0.
\end{equation}
We have for any $0 < \epsilon < 1$,
\[ \sum_k (\nabla_{ik} u)^2 = \sum_k (U_{ik} - \chi_{ik})^2
      \geq (1 - \epsilon) U_{ii}^2 - \frac{C}{\epsilon} .\]
and
\[ \Big(\sum_k \nabla_k u \nabla_{ik} u\Big)^2
       \leq (1 + \epsilon) |\nabla_i u|^2 U_{ii}^2
           + \frac{C}{\epsilon} |\nabla u|^2. \]
Let $\epsilon = \frac{1}{3}$ 
%so that $1 + \epsilon = 2 (1 - \epsilon)$,
and $J = \{i: 2 (n+2) |\nabla_i u|^2 > |\nabla u|^2\}$;
note that $J \neq \emptyset$.
By \eqref{eq4-0-1} and \eqref{eq4-0-3} we obtain
\begin{equation}
\label{eq4-1}
\begin{aligned}
 \frac{1}{3}F^{ii} U_{ii}^2 \,&
    - 2 |\nabla u|^2 \sum_{i \in J} F^{ii} |\nabla_i \phi|^2
   + |\nabla u|^2 (F^{ii} \nabla_{ii} \phi - \phi_t)\\
  \,& \leq C (1 - K_0 |\nabla u|^2) \sum F^{ii} + C |\nabla u|
\end{aligned}
\end{equation}
where $K_0 = \inf_{k,l} R_{klkl}$.

Let 
\[ \phi = - \log (1 - b v^2) + A (\ul{u} + w 
               %- \inf_{\bM_T} (\ul u + w) 
              - B t) \]
where $v$ is a positive function, $A$, $B$ and $b$ are constant, 
all to be determined; $b$ will be chosen sufficiently small 
such that $14 b v^2 \leq 1$ in $\ol{M_T}$, while $A = 0$
in cases (ii) and (iii). 
By straightforward calculations, 
\[ \nabla_i \phi = \frac{2 b v \nabla_i v}{1- b v^2}
      + A \nabla_i (\ul u + w), \;\;
  \phi_t = \frac{2bv v_t}{1-bv^2} + A (\ul u_t  - B) \]
and 
\[ \begin{aligned}
\nabla_{ii} \phi
           = \,& \frac{2 b v \nabla_{ii} v + 2 b |\nabla_i v|^2}{1- b v^2}
                + \frac{4 b^2 v^2 |\nabla_i v|^2}{(1- b v^2)^2}
                 + A \nabla_{ii} (\ul u + w) \\
         = \,& \frac{2 b v \nabla_{ii} v}{1- b v^2}
                + \frac{2 b (1 + b v^2) |\nabla_i v|^2}{(1- b v^2)^2}
                 + A \nabla_{ii} (\ul u + w).
    \end{aligned}\]
% and fix $A > 0$ sufficiently small.
Plugging these into \eqref{eq4-1}, we obtain
\begin{equation}
\label{eq4-2}
\begin{aligned}
 \frac{1}{3} F^{ii} \,& U_{ii}^2
  +  |\nabla u|^2 \sum_{i \in J}  F^{ii} \Big(\frac{b (1 -7 b v^2) 
       |\nabla_i v|^2}{(n+2) (1- b v^2)^2} - C A^2\Big) \\
  +\,& \frac{2 b v |\nabla u|^2}{1- b v^2}( F^{ii} \nabla_{ii} v - v_t )
   + A |\nabla u|^2 (F^{ii} \nabla_{ii} (\ul u + w) - \ul u_t  +  B) \\
   \leq \,& C (1 - K_0 |\nabla u|^2) \sum F^{ii} + C |\nabla u|.
\end{aligned}
\end{equation}
%We now fix $b > 0$ sufficiently small such that $14 b v^2 \leq 1$ in $\bM_T$. 

In both cases (ii) and (iv) we take 
\[ v = \ul u - u + \sup_{\bM_T} (u - \ul u) + 1 \geq 1. \]
Let ${\mu} = \lambda (\nabla^2 \ul u (x_0, t_0)+  \chi (x_0))$,
$\lambda = \lambda (\nabla^2 u (x_0, t_0) + \chi (x_0))$ and $\beta$
as in \eqref{eq2-16}. Suppose first that $|\nu_{\mu}- \nu_{\lambda}|
\geq \beta$. By Lemma~\ref{ma-lemma-C10} and the assumptions 
that $\sum f_i \lambda_i \geq 0$ and $\nabla^2 w \geq \chi$ 
we see that,
\[ F^{ii} \nabla_{ii} (\ul u + w) - \ul u_t  +  B
   \geq %F^{ii} \nabla_{ii} (\ul u - u) - \ul u_t  +  u_t
F^{ii} \nabla_{ii} v - v_t  + (B - u_t)
%= F^{ii} \nabla_{ii} (\ul u - u) - (\ul u - u)_t
    \geq \varepsilon \sum F^{ii} + \varepsilon + (B - u_t)\]
for some $\varepsilon > 0$. Let $A = A_1 K_0^-/\varepsilon$, $K_0^-
= \max \{-K_0, 0\}$ and fix $A_1$, $B$ sufficiently large. A bound
$|\nabla u| \leq C$ follows from \eqref{eq4-2} in both cases 
({ii}) and ({iv}).

We now consider the case $|\nu_{\mu}- \nu_{\lambda}| <  \beta$. 
By \eqref{eq2-22} and \eqref{eq4-2} we see that if $|\nabla u|$ is
sufficiently large,
\begin{equation}
\label{eq4-3}
\begin{aligned}
 \frac{\beta}{\sqrt{n}}  ( |\lambda|^2 + c_1 |\nabla u|^4) \sum F^{ii} 
\,& \leq 
 F^{ii} U_{ii}^2 + 2 c_1 |\nabla u|^4 \sum_{i \in J} F^{ii} \\
\,& \leq C (1 - K_0 |\nabla u|^2) \sum F^{ii} + C |\nabla u|
\end{aligned}
\end{equation}
where $c_1 > 0$.

%If $J=\emptyset$, we use \eqref{eq4-2-1} instead to get

Suppose $|\lambda| \geq R$ for $R$ sufficiently large. Then
\begin{equation}
\label{eq4-4}
\begin{aligned}
%F^{ii} U_{ii}^2 + c_1 |\nabla u|^4 \sum_{i \in J} F^{ii}  
%\,& \geq
 \frac{\beta}{\sqrt{n}}  ( |\lambda|^2 + c_1 |\nabla u|^4) \sum F^{ii} \geq \frac{2 \beta  |\lambda| \sqrt{c_1}}{\sqrt{n}}  |\nabla u|^2 \sum F^{ii} \, & \geq c_2 |\nabla u|^2
 \end{aligned}
\end{equation}
%\[ F^{ii} U_{ii}^2 \geq b_0 |\lambda| \geq b_0 R \]
for some uniform $c_2 > 0$.  
We obtain from \eqref{eq4-3} and \eqref{eq4-4} 
a bound for $|\nabla u (x_0, t_0)|$. 

Suppose now that $|\lambda| \leq R$. Then $\sum F^{ii}$ has a 
positive lower bound by \eqref{eq3-155} and \eqref{eq3-16}. 
Therefore a bound  $|\nabla u (x_0, t_0)|$ follows from
\eqref{eq4-3} again. This completes the proof of \eqref{gradient}
in cases (ii) and (iv). 

For case ({iii}) we choose $A = 0$ and 
\begin{equation}
\label{eq4-0-0}
\phi = ( u  - \inf\limits_{\ol{M_T}}  u  +1)^2.
\end{equation}
By \eqref{eq4-2} 
\begin{equation} 
\label{eq4-0-6}
|\nabla u|^4 \sum_{i \in J}  F^{ii} 
   \leq C (1 - K_0 |\nabla u|^2) \sum F^{ii} + C |\nabla u|.
\end{equation}
By \eqref{eq4-0-1} we see that 
$U_{ii} \leq 0$ for each $i \in J$ if $|\nabla u| $ is
sufficiently large, and a bound 
for $|\nabla u (x_0,t_0)|$ therefore follows from \eqref{eq4-0-6} 
and assumption \eqref{eq1-20}.

\bigskip

\section{Appendix: Proof of Lemma~\ref{ma-lemma-C20}}
\setcounter{equation}{0}

In this Appendix we present a proof of 
Lemma~\ref{ma-lemma-C20} (Theorem \ref{3I-th3}) for the reader's
convenience. The basic ideas of the proof are adopted from \cite{Guan14}.  

For $\sigma \in \bfR$ define
\[ \Sigma^{\sigma} := \{(\lambda, p)  \in \Gamma \times \bfR:
                                f (\lambda)-p > \sigma\}.  \]
Let $\partial \Sigma^{\sigma}$ be the boundary of 
$\Sigma^{\sigma}$ and $T_{\hat{\lambda}}\partial \Sigma^{\sigma}$ denote the tangent hyperplane to $\partial \Sigma^{\sigma}$ 
at $\hat{\lambda} \in \partial \Sigma^{\sigma}$. 
The unit normal vector to $\partial \Sigma^{\sigma}$ at $\hat{\lambda}$ is given by 
\[ \nu_{\hat{\lambda}}
   = \frac{(Df (\lambda), -1))}{\sqrt{1 + |Df (\lambda)|^2}}. \]
In addition, for $\hat \mu \in \Gamma \times \bfR$ let
\[
\begin{aligned}
\hat S^{\sigma}_{\hat{\mu}}
       := &\ \{\hat{\lambda} \in \partial \Sigma^{\sigma}:
(\hat{\mu} - \hat{\lambda}) \cdot \nu_{\hat{\lambda}} \leq 0\},\\
 \mathcal{B}_{\sigma}^+
       := &\ \{\hat{\mu} \in \Gamma \times \bfR:
          \hat S^{\sigma}_{\hat{\mu}} \cap \Gamma \times \{a\}
  \; \mbox{is compact}, \; \forall \, a \in \bfR\},\\
V^{\sigma} := &\ \mathcal{B}_{\sigma}^+\setminus \Sigma^{\sigma},
\end{aligned}
\]
and for $\hat \mu\in V^{\sigma}$,
$$\mathcal{B}_{\sigma}^+(\hat \mu)=\{ t\hat\lambda+(1-t)\hat \mu: 
   \hat\lambda\in \hat S^{\sigma}_{\hat\mu}, 0\leq t\leq 1\}.$$
For convenience we shall write $\hat\lambda = (\lambda, p)$, 
$\hat\mu = (\mu, q)$ and $\tilde{f} (\hat\lambda) = f (\lambda) - p$ in this section.

\begin{lemma}
\label{L2.6}
Let $\delta > 0$, $\hat \mu\in V^{\sigma}$. Then
\[ H_{\hat \mu}(R):= \min_{\hat\lambda\in \partial\Sigma^{\sigma}\cap\{|\lambda|=R, 
 |p-q|\leq\delta\}}(\hat\mu-\hat\lambda)\cdot \nu_{\hat\lambda} > 0  \]
for 
\[ R > R_{\hat{\mu}} :=
    \max_{(\lambda,p)\in \hat S^{\sigma}_{\hat \mu}\cap \Gamma\times [q-\delta,q+\delta]} |\lambda|. \]
\end{lemma}

\begin{lemma}
\label{L2.7}
Let $\hat\mu\in V^{\sigma}$. Then $\mathcal{B}_{\sigma}^+ (\hat\mu) \subset V^{\sigma}$ and
$\mathcal{B}_{\sigma}^+ (\hat\mu') \subset \mathcal{B}_{\sigma}^+ (\hat\mu)$ 
for $\hat\mu' \in \mathcal{B}_{\sigma}^+ (\hat\mu)$.
\end{lemma}

\begin{proof}
Let $\hat\mu_t = t \hat\lambda + (1-t) \hat\mu$ for 
$t \in [0, 1]$ and $\hat\lambda \in \hat S^{\sigma}_{\hat\mu}$.
Since %$(\hat\mu - \hat\zeta) \cdot \nu_{\hat\zeta}> 0$ and
$\partial \Sigma^{\sigma}$ is convex, 
\[ \begin{aligned}
(\hat\mu_t - \hat\zeta) \cdot \nu_{\hat\zeta}
    = \,& (1-t) (\hat\mu -\hat\zeta) \cdot \nu_{\hat\zeta}
       + t (\hat\lambda - \hat\zeta) \cdot \nu_{\hat\zeta} \\
    > \,& t (\hat\lambda - \hat\zeta) \cdot \nu_{\hat\zeta} \geq 0,
\;\; \forall \, \hat\zeta \in \partial \Sigma^{\sigma} \setminus \hat S^{\sigma}_{\hat\mu}.
  \end{aligned} \]
This shows $\hat S^{\sigma}_{\hat\mu_t} \subset
\hat S^{\sigma}_{\hat\mu}$ and therefore $\hat\mu_t \in V^{\sigma}$.
Clearly
$\mathcal{B}_{\sigma}^+(\hat\mu) \subset \mathcal{B}_{\sigma}^+ (\hat\mu)$.
\end{proof}

\begin{lemma}
\label{L2.8}
The cone $\mathcal{B}_{\sigma}^+$ is open.
\end{lemma}

\begin{proof}
%It is enough to show $V^{\sigma}$ is open.
Let $\hat\mu \in V^{\sigma}$
and $a \in \bfR$.
Since $\hat S_{\hat\mu}^{\sigma}\cap \Gamma\times\{a\}$
is compact, 
$\hat S_{\hat\mu}^{\sigma}\cap \Gamma\times \{a\} \subset
B_R\times\{a\}$ for sufficiently large $R$. Therefore  
\[ \alpha := \frac{1}{2 \sqrt{n}} \min_{\hat\zeta \in \partial \Sigma^{\sigma} \cap \partial B_R\times\{a\}}
(\hat\mu - \hat\zeta) \cdot \nu_{\hat\zeta} > 0\]
and 
\[ \begin{aligned}
(\hat\mu - \alpha {\bf \hat 1} - \hat\lambda) \cdot \nu_{\hat\lambda}
   \geq\, & - \sqrt{n} \alpha +
           (\hat\mu - \hat{\lambda}) \cdot \nu_{\hat\zeta}
   \geq \sqrt{n} \alpha > 0,
  \;\; \forall \, \hat\lambda \in \partial \Sigma^{\sigma} \cap \partial B_R\times\{a\}
  \end{aligned} \]
where ${\bf \hat 1} = ({\bf 1}, 0)\in \bfR^{n+1}$. This proves that
$\hat S^{\sigma}_{(\hat\mu - \alpha \hat{\bf 1})} \cap \Gamma
\times\{a\}\subset \partial \Sigma^{\sigma} \cap B_R\times\{a\}$ 
and hence $\hat\mu - \alpha \hat{\bf 1} \in V^{\sigma}$.
On the other hand, for any $\hat\lambda \in \hat S^{\sigma}_{\hat\mu}\cap\Gamma\times\{a\}$, 
\[   \begin{aligned}
(\hat\mu - \alpha {\bf \hat 1} - \hat\lambda)  
     \cdot \nu_{\hat\lambda}
     \leq\, & - \alpha {\bf \hat 1} \cdot \nu_{\hat\lambda}
        = - \frac{\alpha \sum f_i (\lambda)}{\sqrt{1 +|D f|^2}}
        < 0.
  \end{aligned} \]
So $\mathcal{B}_{\sigma}^+(\hat\mu)
\subset \mathcal{B}_{\sigma}^+ (\hat \mu - \alpha {\bf \hat1}) \subset V^{\sigma}$.
Clearly $\mathcal{B}_{\sigma}^+ (\hat \mu - \alpha {\bf \hat 1})$
contains a ball centered at $\hat\mu$.
\end{proof}

\begin{lemma}
\label{L2.9}
Let $K$ be a compact subset of $V^{\sigma}=\mathcal{B}_{\sigma}^+\setminus \Sigma^{\sigma}.$ Then, for any $\delta >0$,
$$\sup_{\hat\mu\in K}\max_{\hat\lambda\in\hat S^{\sigma}_{\hat\mu}\cap\Gamma\times[a-\delta,b+\delta]} |\lambda|< \infty.$$
where $a=\min\{q| \hat\mu \in K\}$, $b=\max\{q| \hat\mu\in K\}$.
\end{lemma}

\begin{proof}
Suppose this is not true. Then for each integer $k \geq 1$
there exists $\hat\mu_k \in K$ and $\hat\lambda_k \in  \hat S^{\sigma}_{\hat\mu_k}\cap\Gamma\times[a-\delta,b+\delta]$
with $|\lambda_k| \geq k$.
By the compactness of $K$
we may assume $\hat\mu_k \rightarrow \hat\mu \in K$ as $k \rightarrow \infty$.
Thus
\[ \limsup_{k \rightarrow \infty} (\hat\mu - \hat\lambda_k) \cdot \nu_{\hat\lambda_k}
   = \limsup_{k \rightarrow \infty} (\hat\mu - \hat\mu_k) \cdot \nu_{\hat\lambda_k} + \limsup_{k \rightarrow \infty} 
(\hat\mu_k - \hat \lambda_k) \cdot \nu_{\hat\lambda_k} \leq 0. \]
On the other hand, by Lemma~\ref{L2.6} there exists $\epsilon > 0$
such that
\[ (\hat\mu - \hat\lambda_k) \cdot \nu_{\hat\lambda_k} \geq \epsilon,
\;\; \forall \, k > \max_{\hat\lambda \in \hat S_{\hat\mu}^{\sigma}\cap\Gamma\times[q-\delta,q+\delta]} |\lambda|. \]
This is a contradiction.
\end{proof}

Let $\hat\mu \in \ol{\Sigma^{\sigma}}$ and
$\hat\lambda \in \partial \Sigma^{\sigma}$. By the convexity of
$\partial \Sigma^{\sigma}$, the open segment
\[ (\hat\mu, \hat\lambda) := \{t \hat\mu + (1-t)\hat\lambda: 0 < t < 1\} \]
is completely contained in either $\partial \Sigma^{\sigma}$ or $\Sigma^{\sigma}$ by condition~\eqref{eq1-3}. Therefore,
\[ \widetilde{f} ( t\hat\mu + (1-t)\hat\lambda) > \sigma, \;\; \forall \; 0 < t < 1 \]
 unless $(\hat\mu, \hat\lambda) \subset \partial \Sigma^{\sigma}$ which is
equivalent to $\hat S^{\sigma}_{\hat\mu} =\hat S^{\sigma}_{\hat\lambda}$.

For $\hat\mu \in \ol{\Sigma^{\sigma}}$, 
$\delta >0$ and $R>|\mu|$ let
$$ \Theta_R(\hat\mu):=\inf_{\hat\lambda\in \{|\lambda|=R, |p-q|\leq \delta  \}\cap \partial \Sigma^{\sigma}}
\max_{0\leq t\leq 1} \widetilde{f}(t\hat\mu+(1-t)\hat\lambda)-\sigma\geq 0.$$
Clearly $\Theta_R(\hat\mu)=0$ if and only if $(\hat\mu, \hat\lambda)\subset \partial \Sigma^{\sigma}$
for some $\hat\lambda\in \{|\lambda|=R, |p-q|\leq \delta  \}\cap \partial \Sigma^{\sigma}$,
since the set $\{|\lambda|=R, |p-q|\leq \delta  \}\cap \partial\Sigma^{\sigma}$ is compact.

\begin{lemma}\label{L2.11}
For $\hat\mu\in \overline{\Sigma^{\sigma}},\Theta_R(\hat\mu)$ is nondecreasing in $R$.
Moreover, if $\Theta_{R_0}(\hat\mu)>0$ for some $R_0\geq |\mu|
$ then $\Theta_{R^\prime}(\hat\mu)>\Theta_R(\hat\mu)$ for all $R^\prime>R\geq R_0$.
\end{lemma}

\begin{proof}
We shall write
$\Theta_R = \Theta_R (\hat\mu)$ when there is no possible confusion.
Suppose $\Theta_{R_0} (\hat\mu) > 0$ for some $R_0 \geq |\mu|$.
Let $R' > R \geq R_0$ and assume that $\Theta_{R'}$ is achieved at
$\hat\lambda_{R'} \in \{\hat\lambda\in \Gamma\times \bfR | |\lambda|=R', |p-q|\leq \delta \}
  \cap \partial \Sigma^{\sigma}$, 
that is, 
\[     \Theta_{R'} =
        \max_{0 \leq t \leq 1} \widetilde{f} (t \hat\mu + (1-t) \hat\lambda_{R'})-\sigma. \]
%In fact, $p_{R'}=q+\delta$.
Let $P$ be the (two dimensional) plane through $\hat\mu, \hat\lambda_{R'}$ and the
point of $(0,q)$. There is a point $\hat\lambda_R \in \{|\lambda|=R, |p-q|\leq \delta  \}$ 
which lies
%between $\mu$ and $\lambda_{R'}$
on the curve $P \cap \Sigma^{\sigma}$.
Note that $\hat\mu$, $\hat\lambda_{R}$ and $\hat\lambda_{R'}$ are not on a straight line,
for $(\hat\mu, \hat\lambda_R)$ can not be part of 
$(\hat\mu, \hat\lambda_{R'})$
since $\Theta_{R_0} > 0$ and $\partial \Sigma^{\sigma}$ is convex.
We see that
\[ \max_{0 \leq t \leq 1} \widetilde{f} 
    (t \hat\mu + (1-t) \hat\lambda_R) -\sigma < \Theta_{R'} \]
by condition \eqref{eq1-3}. This proves $\Theta_{R} < \Theta_{R'}$.
\end{proof}

\begin{lemma}\label{L2.12}
Let $\hat\mu\in \partial\Sigma^{\sigma}\cap \mathcal{B}_{\sigma}^+$ and $\delta>0$.
Then 
$$\Theta_R (\hat\mu) > 0,
\; \forall \, R>\max_{\hat\lambda\in \hat S^{\sigma}_{\hat\mu}
   \cap \Gamma \times [q-\delta,q+\delta]} |\lambda|.$$
\end{lemma}

\begin{proof}
This is obvious.
\end{proof}

\begin{lemma}
\label{L2.14}
For $\hat\mu\in \partial\Sigma^{\sigma}\cap \mathcal{B}_{\sigma}^+$ and $\delta>0$,
let $N_{\hat\mu}=2\max_{\hat\lambda\in \hat S^{\sigma}_{\hat\mu}\cap \Gamma\times[q-\delta,q+\delta]}|\lambda|$.
Then for any $\hat\lambda\in \partial\Sigma^{\sigma}$ with $|p-q|\leq \delta$, when $|\lambda|\geq N_{\hat\mu}$,
\begin{equation}
\sum f_i(\lambda)(\mu_i-\lambda_i)-(q-p)\geq \Theta_{N_{\hat\mu}}(\hat\mu)>0.
\end{equation}
\end{lemma}

\begin{proof}
By the concavity of $\widetilde{f}$ with respect to $\hat\lambda$,
\[ \begin{aligned}
&\sum f_i (\lambda) (\mu_i - \lambda_i)-(q-p)\\
    \geq\, & \max_{0 \leq t \leq 1} \widetilde{f}(t \hat\mu + (1-t) \hat\lambda)-\sigma,
    \;\;  \forall \, \hat\mu, \hat\lambda \in \partial \Sigma^{\sigma}.
\end{aligned}  \]
 %$    - \sigma \geq \Theta_R (\mu) > 0 \]
So Lemma~\ref{L2.14} follows from
Lemma~\ref{L2.11} and Lemma~\ref{L2.12}.
\end{proof}

\begin{lemma}
\label{L2.15}
Let $K$ be a compact subset of $\partial\Sigma^{\sigma}\cap \mathcal{B}^+_{\sigma}$ and $\eta>0$. Define
$$ a=\min\{q|\hat\mu\in K\},\;
b=\max\{q|\hat\mu\in K\}, \;
\delta=|b-a|+\eta$$
 $$N_K:=\sup_{\hat\mu\in K} N_{\hat\mu}, \; 
\Theta_K :=\inf_{\hat\mu\in K}\Theta_{N_K}\hat\mu.$$
Then for any $\hat\lambda \in \partial\Sigma^{\sigma}\cap \Gamma\times[a-\eta,b+\eta]$, when $|\lambda|\geq N_K$,
\begin{equation}\label{2.2}
\sum f_i(\lambda)(\mu_i-\lambda_i)-(q-p)\geq \Theta_K>0, \;
\forall \, \hat\mu\in K.
\end{equation}
\end{lemma}

\begin{proof}
From Lemma~\ref{L2.9} we see that $N_K < \infty$ and consequently,
\[ \Theta_K := \inf_{\hat\mu \in K}
        \inf_{\hat\lambda \in \{|\lambda|={N_K}, |p-q|\leq \delta\}\cap \partial \Sigma^{\sigma}}
        \max_{0 \leq t \leq 1} \widetilde{f} (t \hat\mu + (1-t)\hat\lambda)-\sigma > 0 \]
by the continuity of $\widetilde{f}$.
Now \eqref{2.2} follows from
Lemma~\ref{L2.14}.
\end{proof}

\begin{theorem}
\label{T2.16}
Let  $\delta > 0$, $\hat\mu = (\mu, p) \in \mathcal{B}_{\sigma}^+$ and $0< \varepsilon < \frac12 \mbox{dist} \, (\hat\mu,\partial \mathcal{B}_{\sigma}^+)$.
There exist positive constants $\theta_{\hat\mu}$, $R_{\hat\mu}$
such that for any 
$\hat\lambda = (\lambda, q) \in \partial\Sigma^{\sigma}$ with 
$|p-q|\leq \delta$, when $|\lambda|\geq R_{\hat\mu}$,
\begin{equation}
\label{2.3}
\sum f_i(\lambda)(\mu_i-\lambda_i) -(q-p) \geq \theta_{\hat\mu}+\varepsilon\sum f_i(\lambda).
\end{equation}
\end{theorem}

\begin{proof}
We first note that 
$\hat\mu^{\varepsilon} := \hat\mu - \varepsilon {\bf{\hat 1}} \in \mathcal{B}_{\sigma}^+$.
If $\hat \mu^{\varepsilon} \in\overline{\Sigma^{\sigma}}$
then $\Theta_{R_0}(\hat \mu^{\varepsilon}) > 0$ for some $R_0>0$.
By Lemma \ref{L2.14} we see
%$$\sum f_i(\lambda)(\mu_i-\varepsilon -\lambda_i)-(q-p)
%\geq \Theta_{R_0}(\mu^{\varepsilon},q), \;\forall \,
%\hat\lambda \in \partial\Sigma^{\sigma}\cap\{|\lambda|\geq R_0, 
%|p-q|\leq \delta\}$$
\eqref{2.3} hold when $|\lambda|\geq R_0$.

Suppose now that $\hat\mu^{\varepsilon} \in \mathcal{B}_{\sigma}^+\setminus \overline{\Sigma^{\sigma}}$
and $\hat\lambda \in \partial\Sigma^{\sigma} \setminus 
\hat S^{\sigma}_{\hat \mu^{\varepsilon}}$
with $|p- q| \leq \delta$. 
the segment $(\hat \mu^\varepsilon, \hat\lambda)$ goes through $\hat S^{\sigma}_{\hat \mu^{\varepsilon}} \cap \Gamma \times [q-\delta, q+\delta]$
at a point $\hat \lambda^\prime$.
By the concavity of $\tilde{f}$ and Lemma \ref{L2.15} applied to $K=\hat S^{\sigma}_{\hat \mu^{\varepsilon}} \cap \Gamma \times [q-\delta, q+\delta]$,
we obtain
%for any $\hat\lambda \in \partial\Sigma^{\sigma}$ 
%with $|p-q|\leq \delta$,
$$ \sum \tilde f_i (\hat \lambda) 
           (\hat \mu^\varepsilon_i - \hat \lambda_i) 
   \geq \sum \tilde f_i (\hat \lambda) 
     (\hat \lambda^\prime_i - \hat \lambda_i) \geq \Theta_K>0. $$
This proves \eqref{2.3} for
$\theta_{\hat\mu}=\min\{\Theta_{R_0},\Theta_K\}, R_{\hat\mu}=\max\{R_0, N_K\}$.
\end{proof}

\begin{theorem}
\label{3I-th3}
Let $K$ be a compact subset of $\Gamma \times \bfR$,
$\eta > 0$, 
and let $a$, $b$, $\delta$ be defined as in Lemma~\ref{L2.15}.
Suppose that 
$\hat S^{\sigma}_{\hat{\mu}} [a, b]
  := \hat S^{\sigma}_{\hat{\mu}} \cap \Gamma \times [a, b]$
is compact for any $\hat \mu\in K$.
Then there exist $\varepsilon,\theta_K, R_K > 0$ such that for any $\hat\lambda \in
\partial \Sigma^{\sigma}\cap \Gamma\times[a-\eta,b+\eta]$, when $|\lambda| \geq R_K$,
\begin{equation}
\label{3I-100}
\sum f_i (\lambda) (\mu_{i} - \lambda_i)-(q-p)
    \geq \theta_K + \varepsilon \sum f_i (\lambda),
\;\; \forall \, \hat\mu \in K.
\end{equation}
Furthermore, for any closed interval $[c,d]$,  $\theta_K$ and $R_K$ can be chosen
so that \eqref{3I-100} holds uniformly in $\sigma\in [c,d]$.
\end{theorem}

\begin{proof}
Let $K_1 = \{\hat\mu \in K: \hat \mu^{3 \varepsilon/2} \in \Sigma^{\sigma}\}$, 
$K_2 = \{\hat\mu \in K: \hat\mu^{3 \varepsilon/2} \in V^{\sigma}\}$
and
\[ W := \cup_{\hat \mu \in K_2} 
           \hat S^{\sigma}_{\hat \mu^{3 \varepsilon/2}}. \]
%\tilde{K} = \{\mu^{3 \varepsilon/2}: \mu \in K_2\}. \]
By the concavity of $\tilde{f}$ and compactness on $\ol{K_1}$  we have
\begin{equation}
\label{3I-100a}
\begin{aligned}
\sum \tilde f_i (\hat\lambda)  
             (\hat\mu^{\varepsilon}_i - \hat\lambda_i) 
  \geq \,& \tilde f (\hat\mu^{\varepsilon}) - \tilde f (\hat\lambda) \\
  \geq \,& \min_{\hat\zeta \in \ol{K_1}}  
     \tilde f (\hat \zeta^{\varepsilon})  - \sigma> 0, \;\; \forall \, \hat\mu \in K_1, \; \hat\lambda \in \partial \Sigma^{\sigma}.
\end{aligned}
\end{equation}

Next, 
%\[ \tilde{K} := \{(\mu^{3 \varepsilon/2},q): \hat\mu \in K_2\}, \]
by Lemma~\ref{L2.9},
\[ R_0 := \sup_{(\zeta,r) \in W \cap \Gamma \times [a,b]} |\zeta| < \infty. \]
% \max_{\lambda \in S^{\sigma}_{\zeta}} |\lambda| < \infty. \]
So $\bar{W}\cap\Gamma\times[a,b]$ is a compact subset of
$\mathcal{B}_{\sigma}^+ \cap \partial \Sigma^{\sigma}$.
Applying Lemma~\ref{L2.15} to $\bar{W}$, we obtain for any $\hat\lambda \in \partial \Sigma^{\sigma}\cap\Gamma\times[a,b]$
with $|\lambda| \geq 2 R_0$,
\begin{equation}
\label{3I-100b}
\begin{aligned}
\sum \tilde f_i (\hat\lambda)  
             (\hat\mu^{\varepsilon}_i - \hat\lambda_i) 
  \geq \, & \min_{\hat \zeta \in \bar{W}\cap\Gamma\times[a,b]}
           \sum \tilde f_i (\hat \lambda) (\hat \zeta_i - \lambda_i) 
    \geq \Theta_{\bar{W}},
\;\; \forall \, \hat\mu \in K_2
\end{aligned}
\end{equation}
since the segment $(\hat \mu^{3 \varepsilon/2}, \hat\lambda)$ must intersect
$W\cap\Gamma\times[a,b]$.
Now \eqref{3I-100} follows from \eqref{3I-100a} and \eqref{3I-100b}.

Finally, we note that $\theta_K$ and $R_K$ can be chosen so that they
continuously depends on $\sigma$. This can be seen from the fact that the hypersurface
$\{\partial \Sigma^{\sigma}:\sigma\in [c,d]\}$ form a smooth foliation of the region bounded
by $\partial \Sigma^{c}$ and $\partial \Sigma^{d}$ in $\Gamma\times \bfR$,
which also implies that the distant function dist$(\hat\mu, \partial \mathcal{B}^+_{\sigma})$ also depends continuously on $\sigma$.
\end{proof}

\end{document}